\documentclass{gtmon_a}
\pdfoutput=1

\proceedingstitle{The Zieschang Gedenkschrift}
\conferencestart{5 September 2007}
\conferenceend{8 September 2007}
\conferencename{Conference in honour of Heiner Zieschang}
\conferencelocation{Toulouse, France}

\editor{Michel Boileau}
\givenname{Michel}
\surname{Boileau}

\editor{Martin Scharlemann}
\givenname{Martin}
\surname{Scharlemann}

\editor{Richard Weidmann}
\givenname{Richard}
\surname{Weidmann}

\title[Intersections of conjugates of Magnus subgroups of one-relator groups]{Intersections of conjugates of Magnus subgroups\\of one-relator groups} 
                   
\author{Donald J Collins}                  
\givenname{Donald J}
\surname{Collins}
\address{School of Mathematical Sciences\\
Queen Mary, University of London\\\newline
Mile End Road\\London E1 4NS\\UK}                            
\email{d.j.collins@qmul.ac.uk}                     
\urladdr{}

\volumenumber{14}
\issuenumber{}
\publicationyear{2008}
\papernumber{09}
\startpage{135}
\endpage{171}

\doi{}
\MR{}
\Zbl{}

\arxivreference{} %%% Not in ArXiV


\received{14 March 2006}
\revised{}
\accepted{30 January 2007}  %%%% Check Eds
\published{29 April 2008}
\publishedonline{29 April 2008}
\proposed{}
\seconded{}
\corresponding{}
\version{}

\makeatletter
\def\cnewtheorem#1[#2]#3{\newtheorem{#1}{#3}[section]
\expandafter\let\csname c@#1\endcsname\c@lem}
\makeatother

\AtBeginDocument{\let\bar\wbar\let\tilde\wtilde\let\hat\what}
\AtBeginDocument{\def\S{Section }}
\numberwithin{equation}{section}
\def\strut{\vrule width0pt depth0pt height7.5pt}
\def\strutb{\vrule width0pt depth0pt height6pt}
\def\olazmi{\smash{\overleftarrow{z_m^{\smash{-1}}\strut}}}
\newcommand\olazp[1]{\smash{\overleftarrow{z\smash{\rlap{$'$}}_{#1}\strut}}}
\def\olajpo{\smash{\overleftarrow{\smash{j{+}1}\strutb}}}
\def\olaj{\smash{\overleftarrow{\smash{j}\strutb}}}
\def\oraj{\smash{\overrightarrow{\smash{j}\strutb}}}
\def\orai{\smash{\overrightarrow{\mskip-1mu\smash{i}\strutb}}}

\newtheorem{lem}{Lemma}[section]          % Lemma environment with numbering 
\newtheorem*{red}{Reduction}          % Unnumbered environment for remarks.
\newtheorem*{sclm}{Subclaim}
\newtheorem*{llem}{Lemma}
\newtheorem*{slem}{Sublemma}
\newtheorem{clm}{Claim}[section] 
\cnewtheorem{prop}[lem]{Proposition}
\newtheorem{Thm}{Theorem}
\makeautorefname{Thm}{Theorem}
\theoremstyle{definition}
\cnewtheorem{defn}[lem]{Definition}    % Definition environment with 
\newtheorem{add}{Addendum}
\newtheorem{caa}{Case Assumption}

\makeautorefname{caa}{Case Assumption}
 
\begin{document}

\begin{asciiabstract} 
In the theory of one-relator groups, Magnus subgroups, which are free subgroups obtained by omitting a generator that occurs in the given relator, play an essential structural role. In a previous article, the author proved that if two distinct Magnus subgroups M and N of a one-relator group, with free bases S and T are given, then the intersection of M and N is  either the free subgroup P generated by the intersection of S and T or the free product of P with an infinite cyclic group.  

The main result of this article is that if M and N are Magnus subgroups (not necessarily distinct) of a one-relator group G  and g and h are elements of G, then either the intersection of gMg^{-1} and hNh^{-1} is cyclic (and possibly trivial), or gh^{-1} is an element of NM in which case the intersection is a conjugate of the intersection of M and N.
\end{asciiabstract}

\begin{htmlabstract}
<p class="noindent">
In the theory of one-relator groups, Magnus subgroups, which are free subgroups obtained by omitting a generator that occurs in the given relator, play an essential structural role. In a previous article, the author proved that if two distinct Magnus subgroups M and N of a one-relator group, with free bases S and T are given, then the intersection of M and N is  either the free subgroup P generated by the intersection of S and T or the free product of P with an infinite cyclic group.
</p>
<p class="noindent">
The main result of this article is that if M and N are Magnus subgroups (not necessarily distinct) of a one-relator group G  and g and h are elements of G, then either the intersection of gMg<sup>-1</sup> and hNh<sup>-1</sup> is cyclic (and possibly trivial), or gh<sup>-1</sup> is an element of NM in which case the intersection is a conjugate of the intersection of M and N.
</p>
\end{htmlabstract}

\begin{abstract}
In the theory of one-relator groups, Magnus subgroups, which are free subgroups obtained by omitting a generator that occurs in the given relator, play an essential structural role. In a previous article, the author proved that if two distinct Magnus subgroups $M$ and $N$ of a one-relator group, with free bases $S$ and $T$ are given, then the intersection of $M$ and $N$ is  either the free subgroup $P$ generated by the intersection of $S$ and $T$ or the free product of $P$ with an infinite cyclic group.  

The main result of this article is that if $M$ and $N$ are Magnus subgroups (not necessarily distinct) of a one-relator group $G$  and $g$ and $h$ are elements of $G$, then either the intersection of $gMg^{-1}$ and $hNh^{-1}$ is cyclic (and possibly trivial), or $gh^{-1}$ is an element of $NM$ in which case the intersection is a conjugate of the intersection of $M$ and $N$.
\end{abstract}

\maketitle

%%%%%%%%%%%%%%%%%%%%   Start of main body of article

\section{Introduction}\label{sec1} A Magnus subgroup of a one-relator group $G=\langle X : r=1 \rangle$, where $r$ is cyclically reduced, is a subgroup generated by a Magnus subset $S$ of $X$, ie a subset $S$ which omits a generator explicitly occurring in the relator $r$.  By the Freiheitssatz of Magnus (see for example page 104 or page 198 of Lyndon and Schupp \cite{LS}), any such subgroup is free with the given subset as basis.

The classical proof of many theorems on one-relator groups is by induction on the length of the relator. In its modern form, the inductive step in the classical proof expresses a one-relator group $G$ as an HNN-extension of a one-relator base group $G^*$ where the edge subgroups are Magnus subgroups of $G^*$.   Thus Magnus subgroups play a central role in this approach to the theory of one-relator groups.

In a previous article \cite{C}, we determined the form of the intersection of two Magnus subgroups. The precise statement is:

\begin{Thm}\label{thm1} Let $G = \langle X : r=1 \rangle$, where $r$ is cyclically reduced, be a one-relator group and let $M=F(S),N=F(T)$ be Magnus subgroups of $G$. If $M\cap N$ is distinct from $F(S\cap T)$, then $M\cap N$ is the free product of $F(S\cap T)$ and an infinite cycle.
\end{Thm}

In the present article we examine the intersection of conjugates of two Magnus subgroups, and it suffices to deal with the case of an intersection of the form $gMg^{-1}\cap N$, where $M=F(S),N=F(T)$. A simple and obvious argument shows that if $g\in NM$, then $gMg^{-1}\cap N$ is just a conjugate of $M\cap N$ by an element of $N$ and in particular is isomorphic to $M\cap N$.  Our main conclusion deals with the alternative case.

\begin{Thm}\label{thm2} Let $G = \langle X : r=1 \rangle$, where $r$ is cyclically reduced, be a one-relator group and let $M = F(S),N=F(T)$ be Magnus subgroups of $G$, allowing $M=N$. For any $g \in G$, either $gMg^{-1}\cap N$ is cyclic (possibly trivial) or $g\in NM$.
\end{Thm}

A simple argument also enables one to describe the form of an intersection $gM{g^{\prime}}^{-1} \cap N$, where $M,N$ are Magnus subgroups and $g, g^{\prime} \in G$, in terms of the intersections $gMg^{-1} \cap N$ and $g^{\prime}M{g^{\prime}}^{-1} \cap N$.
 
It is surprising that the questions addressed in \fullref{thm1} and \fullref{thm2} have not been examined more extensively, given that some of the difficulty in studying one-relator groups arises precisely from the situation where a pair of Magnus subgroups have \textit{exceptional} intersection, that is  $F(S)\cap F(T)\neq F(S\cap T)$. However there are some partial results that deal with special cases of \fullref{thm1} and \fullref{thm2}. In particular Bagherzadeh \cite{Ba} has shown that if $M=F(S)$ is a Magnus subgroup and $g \notin M$, then $gMg^{-1}\cap M$ is cyclic (possibly trivial) and in \cite{Br}, Brodski{\u\i} actually considered a more general situation and showed that in a one-relator product $\langle A * B \ |\  r=1\rangle$ of locally indicable groups, the intersections $A \cap B$, $gAg^{-1} \cap A$ and $gAg^{-1} \cap B$ are all cyclic (possibly trivial).  In the context of one-relator groups, Brodski{\u\i}'s results imply that if the Magnus subsets $S$ and $T$ are disjoint, then $F(S) \cap F(T)$, $gF(S)g^{-1}\cap F(S)$ and $gF(S)g^{-1}\cap F(T)$ are cyclic. Finally Newman \cite{N} showed that in one-relator groups with torsion, Magnus subgroups are malnormal, ie if $M=F(S)$, where $S$ is a Magnus subset and $g \notin M$, then $gMg^{-1}\cap M$ is trivial  

In addition, in \cite{C}, we also showed that, by extending a version of Newman's argument, one can easily prove that if the one-relator group $G$ has torsion, ie when the relator is a proper power, then, for any two Magnus subgroups $M=F(S)$ and $N=F(T)$ and any $g \in G$, $M \cap N$ is not exceptional and either $gMg^{-1}\cap N$ is trivial or $g\in NM$. Moreover, Newman's approach -- using the so-called Spelling Lemma -- also yields, in the torsion case, an algorithm to determine the precise form of $gMg^{-1}\cap N$, in particular to determine for a given $g$ whether or not $g \in NM$.   These strong results that follow from Newman's work underline why one-relator groups with torsion are easier to work with than one-relator groups in general.

\fullref{thm1} has been significantly extended and generalised by Howie in \cite{H} where he provides a detailed description of how the \textit{exceptional} case  can arise and generalises \fullref{thm1} to the case of a one-relator product of locally indicable groups.  In addition his methods provide an algorithm to determine for a given one-relator group and two Magnus subgroups $M$ and $N$, whether or not $M\cap N$ is exceptional and to determine a generator for the additional infinite cycle in the exceptional case.

 In contrast to the situation for the intersection of two Magnus subgroups, the algorithmic problems arising from \fullref{thm2} remain open.  The difficulty appears to be caused by the case of two-generator one-relator groups.  For both \fullref{thm1} and \mbox{\fullref{thm2}}, there is nothing to prove in this case, for if $G = \langle a,b \ | \ r=1 \rangle$, then the Magnus subgroups  $M=F(a)$ and $N=F(b)$ are both cyclic. In the case of \fullref{thm1}, the algorithmic determination of $F(a) \cap F(b)$ is provided by a procedure based on the Baumslag--Taylor algorithm for determining the centre.  The methods of \cite{H} then yield a procedure for the general case.  For the case of \fullref{thm2} when $G = \langle a,b \ | \ r=1 \rangle$ one has to be able to determine, for a given $g \in G$, the intersections $gF(a)g^{-1} \cap F(a)$ and $gF(a)g^{-1} \cap F(b)$.  In the latter case, one appears to need, as part of the procedure, to be able to determine whether or not $g \in F(b)F(a)$.  For this additional question, despite the fact that, in his solution to the word problem for one-relator groups, Magnus proved that one can always decide if a given element lies in a given Magnus subgroup, the usual inductive technique seems to run aground in the two-generator case when neither generator has exponent sum zero in the relator.

\begin{add}$\mspace{-16.5mu}$ {\rm to \cite{C}}\qua
The reader of \cite{C} should note that although, in the definition of notation on page 273 of \cite{C}, it is made clear that the sets $A^*_+$ and $A^*_-$ may both be empty and similarly for $C^*_+$ and $C^*_-$, there is no specific discussion in Sections 5--6 of \cite{C} of what happens when these possibilities arise.  However, as we point out below, in practice the results in these sections and the similar results in \fullref{sec5} are employed only in situations where all of $A^*_+, A^*_-, C^*_+, C^*_-$ are nonempty.  This point is clarified in the introduction to \fullref{sec4}. It is also worth pointing out that $B^*$ may be empty -- however since the role of $B^*$ throughout the argument is essentially passive, it is clear that nothing is disturbed if $B^*$ is empty.
\end{add}
\begin{add}$\mspace{-14mu}$ {\rm to \cite{C}}\qua\label{add2}
In Lemma 6.1 on page 286 of \cite{C}, the notation $L$ is used with two distinct meanings, only one of which is explained in the text.  The meaning explained is the one that occurs right throughout the whole of \cite{C}, namely that $L$ denotes the ``lower'' edge group in the representation of our one-relator group $G$ as an HNN-extension, for example as $G= \langle G^*, b \ |\ bLb^{-1}=U \rangle$, where $U$ is the upper edge group.  The second meaning, which is used throughout \S 6 of \cite{C} and in \fullref{sec5}, is to denote by $L(z)$ the syllable length, as defined on page 283 of \cite{C} of an element $z$ of, for instance, $F(A^*_+,B^*,C^*)$.
\end{add}

\section[Structure and simple cases in the proof of \ref{thm2}]{Structure and simple cases in the proof of \fullref{thm2}}\label{sec2}

The proof of \fullref{thm2} proceeds, as is usual, by induction on the length of the relator.  We make various initial reductions and then address three separate cases at the inductive step.  Of these, the first is straightforward and the third reduces easily to the second. However, the second case is complicated and requires substantial analysis, making use of some of the technical results from \cite{C}.

\medskip
\textbf{Initial Observations}
\begin{enumerate}
\item [(i)] For small values of $|r|$, the result is elementary by inspection.
\item [(ii)] The general case will follow, via the normal form theorem for free products,  from the case when $S \cup T = \hbox{Supp}(r)$ and so we can always assume the latter.
\item [(iii)] When $|\hbox{Supp}(r)| = 2$, the conclusion is immediate, so that we can assume that $|\hbox{Supp}(r)| \ge 3$.
\item [(iv)] If $S \subset T$ then $gF(S)g^{-1}\cap F(T) \subseteq gF(T)g^{-1}\cap F(T)$ which, by \cite{Ba}, is cyclic unless $g \in F(T)$ so that we can always assume that $S\cap T$ is a proper subset of both $S$ and $T$;
\item [(v)] (assuming (iv)) If we write $B= S\cap T$ and then  choose $A$ and $C$ disjoint so that $S = A\cup B$ and $T=B\cup C$, then the general case reduces to the case when $A$ and $C$ are singletons, say $A = \{a\}$ and $C=\{c\}$. 
\end{enumerate} 

We therefore take all of these as given and embark upon the inductive case; our strategy will always be to assume the conclusion false and then work our way to a contradiction.  In particular we shall assume that there exist $g \notin F(B,C)F(A,B), h,h' \in F(A,B)$ and $k,k' \in F(B,C)$ such that $\{h,h'\}$ (and, necessarily, $\{k,k'\}$) constitute a free basis of the corresponding subgroup they generate. We shall refer to such a configuration as a \textit{counterpair}.

There are three cases:
\begin{enumerate}
\item [(2.1)] Case Assumption: Either $a$ or $c$ has exponent sum zero in $r$.  
\item [(2.2)] Case Assumption: Neither $a$ nor $c$ has exponent sum zero in $r$ but there exists $b \in B$ such that $\sigma_b(r) =0$.
\item [(2.3)] Case Assumption: No generator has exponent sum zero in $r$. 
\end{enumerate}

\textbf{Case 2.1}\qua Without loss of generality we may assume that $a$ has exponent sum zero in $r$. We may further assume, by replacing $r$ by a cyclic permutation if necessary, that $c^{\pm 1}$ is the initial letter of $r$. 

In the standard manner we can express $G$ as an HNN-extension of the form $G = \langle G^*, a \ |\  aLa^{-1} = U \rangle $ where $L$ and $U$ are Magnus subgroups of the base group $G^*$.  To do this we define $C^* = \{c_{\mu},  \ldots, c_{\nu}\}$ and $B^* = \{b_i, i\in {\bf Z}, b\in B\}$ where, as usual, $b_i$ and $c_i$ denote the conjugates $a^iba^{-i}$ and $a^ica^{-i}$ with $\mu$ and $\nu$ respectively the minimal and maximal subscripts that appear when we rewrite $r$ as a word $r^*$ in $B^*\cup C^*$.  With this notation $G^* = \langle B^*,C^* \ | \ r^* =1 \rangle$ and the two edge groups are $L= F(B^*,C^*_-), U=F(B^*,C^*_+)$, where $C^*_- = \{c_{\mu},  \ldots, c_{\nu-1}\}$ and $C^*_+ = \{c_{\mu+1},  \ldots, c_{\nu}\}$.  We note that by requiring that $r$ begins with $c^{\pm 1}$ we have ensured that $\mu \le 0 \leq \nu$. (We do not exclude the possibility that $\mu=0=\nu$ in which case $C^*_+$ and $C^*_-$ are both empty but we will not usually make explicit reference to this since the argument is either unchanged or even simplified.) Given any $z \in U$ we write $\overleftarrow z$ for the word obtained by reducing subscripts by one and similarly for any $w \in L$, we write $\overrightarrow w$ when we increase the subscripts by one.

We can transform any equality $gh(A,B)g^{-1}=k(B,C)$ into one expressed in the generators of $G$ as HNN-extension. We  write $g=g_0a^{\varepsilon_1}g_1\ldots a^{\varepsilon_m}g_m$ in reduced form, where $\varepsilon_i = \pm 1$. Since $k$ omits $a$, $h$ has zero exponent sum in $a$ and thus both $h$ and $k$ lie in the base group $G^*$ -- $h \in F(B^*)$ and $k \in F(B_0,c_0)$ where $B_0= \{b_0 \ | \ b \in B\}$. Among all counterpairs, we choose one with $m=l_b(g)$ minimal.  If $l_b(g) = 0$, then both equalities hold in the base group $G^*$ and hence we can only have $g \in F(B_0,c_0)F(B^*)$.  But then clearly $g \in F(B,C)F(A,B)$ and we have reached a contradiction, as we wish. 

Suppose, then that $l_a(g)>0$.  Choosing $\varepsilon_m =-1$, just for definiteness, we obtain $g_mhg_m^{-1} = z \in F(B^*,C^*_+)$ and $g_mh'g_m^{-1} = z' \in F(B^*,C^*_+)$. By the induction hypothesis on $|r|$, we can only have $g_m \in F(B^*,C^*_+)F(B^*) =F(B^*,C^*_+)$. Then $$ghg^{-1}= g_0a^{\varepsilon_1}g_1\ldots a^{\varepsilon_{m-1}}g_{m-1}\overleftarrow {g_m}\overleftarrow h\overleftarrow {g_m}^{-1}g_{m-1}^{-1}a^{-\varepsilon_{m-1}}\ldots g_1^{-1}a^{-\varepsilon_1}g_0^{-1} =k$$ and similarly for $h'$ and $k'$.  Since $\overleftarrow h$ and $\overleftarrow {h'}$ are conjugates of $h$ and $h'$, it follows from the minimality of our choice of $m$ that the only conclusion we can have is that $g_0a^{\varepsilon_1}g_1\ldots a^{\varepsilon_{m-1}}g_{m-1}\overleftarrow {g_m}\in F(B,C)F(A,B)$ and from this in turn it follows that $g\in F(B,C)F(A,B)$, which is the required contradiction.

\medskip\textbf{Case 2.2}\qua  To deal with this case we employ techniques similar to those of \S 5 of \cite{C} and follow pages 272--273 of \cite{C} in our notation and terminology. Thus we express $G = \langle X : r=1 \rangle$ as an HNN-extension of the form $G = \langle G^*, b \ |\  bLb^{-1} = U \rangle$ where the following hold. 
\begin{enumerate}\item [(1)] $G^* = \langle X^* \ |\ r^* \rangle$ where $X^*=\{a_{\kappa}, \ldots, a_{\lambda}, c_{\mu}, \ldots, c_{\nu}\}\cup\{x_i, i \in {\bf Z}\}, x \neq a,c$,  $\{a_{\kappa}, a_{\lambda}, c_{\mu}, c_{\nu}\}$ are the respective minimal and maximal generators in $r^*$ associated with $a$ and $c$, and otherwise the subscript range is infinite.
\item [(2)] Furthermore $L= F(A^*_-,B^*,C^*_-)$ and $U = F(A^*_+, B^*, C^*_+)$ where
$$\begin{matrix}A^*= \{a_{\kappa}, \ldots, a_{\lambda}\}, &A^*_+ = \{a_{\kappa+1}, \ldots, a_{\lambda}\}, &A^*_- = \{a_{\kappa}, \ldots, a_{\lambda-1}\},\\
C^* = \{c_{\mu}, \ldots, c_{\nu}\}, &C^*_+ = \{c_{\mu+1}, \ldots, c_{\nu}\}, &C^*_- = \{c_{\mu}, \ldots, c_{\nu-1}\},\end{matrix}$$ and $B^* = \{x_i, x \in B',i \in {\bf Z}\}$ where $B' = B \setminus \{b\}$. 
\end{enumerate}

We allow the possibility that $\kappa = \lambda$ or $\mu = \nu$, or both. If, for example, $\kappa = \lambda$, then $A^*_+$ and $A^*_-$ are empty ; arguments which make reference to these must be interpreted suitably for this case.  Also $B^*$ may be empty but as noted already in reference to \cite{C}, nothing in an argument will be disturbed if in fact $B^*$ is empty.

We employ a subsidiary induction on $l_b(g)$.  The inductive step when $l_b(g)>0$ is comparatively straightforward and we deal with it in \fullref{sec3}. Then in \fullref{sec4} we tackle the core of the argument, namely the case when $l_b(g)=0$.

\medskip\textbf{Case 2.3}\qua As described below, the standard method for dealing with the case when no generator has exponent sum zero in $r$ reduces this case to Case 2.2.  

We have at least three generators $a,b,c$ where $A=\{a\}, b\in B, C=\{c\}$; suppose that $r$ has exponent sum $\alpha \neq 0$ in $a$ and exponent sum $\beta \neq 0$ in $b$.  Introduce new generators $x,y$ setting $b=y^\alpha$ and $a=xy^{-\beta}$ so that we have embedded $G$ in the amalgamated free product $\hat G = \langle G*F(y) \ |\ b=y^\alpha \rangle$ and then replaced $a$ by $x=ay^{\beta}$.  The resulting relator $\hat r \equiv r(xy^{-\beta},y^{\alpha}, \ldots , c)$ has exponent sum zero in $y$ and our equalities become $g\hat h(X,Y)g^{-1}=\hat k(Y,C)$ and $g\hat h'(X,Y)g^{-1}=\hat k'(Y,C)$ where $X=\{x\}, Y=\{y,B'\}, C=\{c\}$ and $B = \{b,B'\}$. 

Now Case 2.2 applies and gives the conclusion that either $gF(X,Y)g^{-1}\cap F(Y,C)$ is cyclic or $g \in F(Y,C)F(X,Y)$. Clearly $F(A,B) \subseteq F(X,Y)$ and $F(B,C) \subseteq F(Y,C)$, so that $gF(A,B)g^{-1}\cap F(B,C) \subseteq gF(X,Y)g^{-1}\cap F(Y,C)$.  By our counterpair assumption, $gF(X,Y)g^{-1}\cap F(Y,C)$ cannot be cyclic and therefore $g \in F(Y,C)F(X,Y)$.  But $G \cap F(Y,C)F(X,Y) = F(B,C)F(A,B)$ and our counterpair assumption rules this out and we have the required contradiction. 

\medskip This completes the logical structure of the proof of \fullref{thm2} but of course it remains to deal with Case 2.2.

\section[Case 2.2: the inductive step when l-b(g)>0]{Case 2.2: the inductive step when $l_b(g) >0$}\label{sec3}

As noted above, Case 2.2 is dealt with by a subsidiary induction on $l_b(g)$.  In this section we deal with the inductive step of this subsidiary induction and hence reduce Case 2.2 to the initial step of the subsidiary induction when $l_b(g) =0$.  

Our standpoint here is, therefore, that we have an overall inductive hypothesis which asserts that the theorem holds for relators of shorter length and, arguing by contradiction, we are assuming that there exist counterpairs $ghg^{-1}=k$ and $gh'g^{-1}=k'$.  For the purposes of the subsidiary induction we know there is a counterpair where $l_b(g)$ is minimal but strictly positive. 

\begin{red} In any such counterpair $l_b(h) = l_b(k)= l_b(h')= l_b(k')=0$. \end{red}
\begin{proof}  Suppose not; without loss of generality, we may assume that $l_b(h)>0$. Let us write $g = \tilde gb^{\varepsilon_m}g_m$ where $m=l_b(g)$ and $l_b(\tilde g)=m-1$.  Still without loss of generality we also assume that $\varepsilon_m = -1$.  Then we can write 
$$ghg^{-1}= \tilde gb^{-1}g_mh_0b^{\zeta_1}h_1 \ldots b^{\zeta_l}h_lg_m^{-1}b\tilde g^{-1} = k.$$ Now two subcases arise depending on whether or not (the detailed expression for) $ghg^{-1}$ is or is not reduced. If the latter occurs, then 
 either $\zeta_1=1$ and $g_mh_0 =z \in U$ or $\zeta_l=-1$ and $h_lg_m^{-1} =z \in U$.  It suffices to assume the first occurs. Substituting for $g_m$ we obtain
$$\tilde gb^{-1}zh_0^{-1}hh_0z^{-1}b\tilde g^{-1} = k,
\tilde gb^{-1}zh_0^{-1}h'h_0z^{-1}b\tilde g^{-1}  = k'$$
and hence $$\tilde g\overleftarrow z(b^{-1}h_0^{-1}hh_0b){\overleftarrow z}^{-1}b\tilde g^{-1} = k, \tilde g\overleftarrow z(b^{-1}h_0^{-1}h'h_0b){\overleftarrow z}^{-1}\tilde g^{-1} = k'.$$
By the minimality of $l_b(g)$, we deduce that conjugates of $h$ and $h'$ commute, which of course is a contradiction, or that $\tilde g\overleftarrow z \in F(B,C)F(A,B)$. However, if the latter holds then $g = \tilde gb^{-1}g_m = \tilde gb^{-1}zh_0^{-1} = \tilde g\overleftarrow z b^{-1}h_0^{-1} \in F(B,C)F(A,B)$ and here too we have the necessary contradiction.

To complete the proof of the Reduction we have to see what happens when our expression for $ghg^{-1}$ is reduced.  Then of course $l_b(k) >0$, say $k = k_0b^{\xi_1}k_1 \ldots b^{\xi_n}k_n$ (where $n= 2m+l$) and $h= g^{-1}kg$, with $g^{-1}kg$ not reduced. But now we can argue exactly as we did when $ghg^{-1}$ was not reduced. \end{proof}

To complete the inductive step in the subsidiary induction, it remains only to show for the case at hand that we cannot have a counterpair when $l_b(h) = l_b(k)= l_b(h')= l_b(k')=0$.  We write $g = \tilde gb^{\varepsilon_m}g_m$ as above, and can take $\varepsilon_m =-1$.  Then we obtain
$\smash{\tilde gb^{-1}g_mhg_m^{-1}b\tilde g^{-1} = k \ , \   \tilde gb^{-1}g_mh'g_m^{-1}b\tilde g^{-1}=k'}$ and hence
$g_mhg_m^{-1}= z \in U, \quad g_mh'g_m^{-1}= z' \in U.$ Now these equalities define elements of the intersection $g_mF(A^*,B^*)g_m^{-1} \cap F(A^*_+,B^*,C^*_+)$, which involves Magnus subgroups of a group with a shorter relator. (Note that we cannot have all of $A^*_+,B^*,C^*_+$ empty since clearly $h,h' \neq 1$.)  Hence either $h$ and $h'$ commute or $g_m \in F(A^*_+,B^*,C^*_+)F(A^*,B^*)$, say $g_m =z_0h_0$.  The former is contradictory and so from the latter we obtain $\smash{z_0h_0hh_0^{-1}z_0^{-1} = z}$ and $\smash{z_0h_0h'h_0^{-1}z_0^{-1} = z'}$ whence $\smash{h_0hh_0^{-1} \in U, h_0h'h_0^{-1} \in U}$.  Let us write $\smash{x= h_0hh_0^{-1}}$ and $\smash{x'= h_0h'h_0^{-1}}$.  Then we have $\smash{\tilde gb^{-1}z_0xz_0^{-1}b\tilde g^{-1}}{=} k$ and $\smash{s \tilde gb^{-1}z_0x'z_0^{-1}b\tilde g^{-1}{=}k'}$. This yields
$\tilde g{\overleftarrow {z_0}}\overleftarrow x\smash{{\overleftarrow {z_0}}}^{\!-1}\tilde g^{-1} {=} k$ and $\tilde gb^{-1}{\overleftarrow {z_0}}\overleftarrow {x\smash{'}\strut}\smash{{\overleftarrow {z_0}}}^{\!-1}b\tilde g^{-1}=k'.$  By the minimality of $l_b(g)$, we deduce that conjugates of $h$ and $\smash{h'}$ commute (either $k$ and $k'$ or $\overleftarrow x$ and $\smash{\overleftarrow {x\smash{'}\strut}}$) or $\smash{\tilde g{\overleftarrow {z_0}}} \in F(B,C)F(A,B)$.  In the latter case $$g = \tilde gb^{-1}g_m = \tilde gb^{-1}z_0h_0 = \tilde g{\overleftarrow {z_0}} h_0\in F(B,C)F(A,B)$$ and we have the required contradiction when $l_b(g) >0$. 

\section[Case 2.2: the case when l-b(g)=0]{Case 2.2: the case when $l_b(g) = 0$}\label{sec4}

Our standpoint is, again, that we have an overall inductive hypothesis which asserts that the theorem holds for relators of shorter length and, arguing by contradiction, we are assuming that there exist counterpairs $ghg^{-1}=k$ and $gh'g^{-1}=k'$.  This time, however, we assume that there is a counterpair where $l_b(g)=0$, ie $g \in G^*$. 

Since $l_b(g) = 0$, $h$ and $k$ have the same signature pattern in $b$ and the same is true for $h'$ and $k'$. Possibly $l_b(h)=0$ and $h\in G^*$ while if $l_b(h) > 0$ then we have a sequence of Normal Form Equalities derived from the equality $$ghg^{-1} \equiv gh_0b^{\varepsilon_1}h_1 \ldots b^{\varepsilon_m}h_mg^{-1} = k_0b^{\varepsilon_1}k_1 \ldots b^{\varepsilon_m}k_m \equiv k$$ in which $h,k$ are expressed in reduced form in terms of the HNN extension $G = \langle G^*, b\ |\  bLb^{-1} = U \rangle$.  A similar observation applies to $h'$.

By Lemma 5.1 of \cite{C}, $h_0, \ldots h_m \in F(A^*,B^*)$ and $k_0, \ldots k_m \in F(B^*,C^*)$.   The Normal Form Equalities for the $ghg^{-1}=k$ are then $$gh_0 = k_0z_0, \overline {z_0}h_1 = k_1z_1, \overline {z_1}h_2 = k_2z_2, \ldots, \overline {z_{m-1}}h_m = k_mg$$ where $z_0, \ldots z_{m-1}$ lie in $L$ or $U$ according as ${\varepsilon_1}, \ldots, {\varepsilon_m}$ are $\pm 1$ and $\overline {z_{i-1}}$ represents a ``downshift'' or ``upshift'' of subscripts according as ${\varepsilon_i} = \pm 1$.  When we have such a sequence, $gh_0b^{\varepsilon_1}h_1 \ldots b^{\varepsilon_i}h_i = k_0b^{\varepsilon_1}k_1 \ldots b^{\varepsilon_i}k_iz_i$ and the fact that  $g \notin F(B,C)F(A,B)$ means that $z_i\notin F(B,C)F(A,B)$  and in particular is nontrivial. Moreover, when we use the Normal Form Equalities in standardised form, as described on pages 275-276 of \cite{C}, the elements $z_i$ will always be nontrivial and of type $(A^*:C^*)$.  In fact if $z_i \in U$ then $z_i$ must actually be of type $(A^*_+:C^*_+)$ in which case both $A^*_+$,  and $C^*_+$ are nonempty, and therefore both $A^*_-$ and $C^*_-$ are nonempty. A similar remark applies if $z_i \in L$.  This means that our applications of the results of Sections 5--6 of \cite{C} and of \fullref{sec5} are applied under the hypotheses that there is no hidden ``collapsing'' of the terms denoted by the notation. 

We shall use these observations throughout this section without further reference. We shall establish a series of claims which, cumulatively, will demonstrate that there are no counterpairs satisfying $l_b(g)=0$.

\begin{clm}\label{Claim 4.1}There do not exist counterpairs $ghg^{-1}=k, gh'g^{-1}=k'$, with $l_b(g)=0$, such that min$\{l_b(h),l_b(h')\} = 0$.
\end{clm}
\begin{proof} Suppose not; clearly there is no loss of generality in assuming $l_b(h) =l_b(k) =0$. Among all such counterpairs, we choose one with $l_b(h') = l_b(k')$ minimal. If $l_b(h') = l_b(k')=0$, then $h,h' \in F(A^*,B^*), k, k' \in F(B^*,C^*)$ and we have an immediate contradiction to the overall induction hypothesis. 

So we can assume that $l_b(h') = l_b(k')>0$. We can write
$$gh'g^{-1} = gh'_0b^{\varepsilon_1}h'_1 \ldots b^{\varepsilon_n}h'_ng^{-1} = k'_0b^{\varepsilon_1}k'_1 \ldots b^{\varepsilon_n}k'_n =k'$$ and there is no loss of generality in assuming that $\varepsilon_1 =1$. Adjusting the equality so that it is standardised form, we obtain   $g=k'_0z_0{h'_0}^{-1}$, where $z_0 \in U$ is nontrivial of type $(A^*:C^*)$. Substituting for $g$ and replacing the original $h,k,h'$ and $k'$ by the resulting conjugates, we can rewrite, adjusting our notation, the two equalities in the form 
$$z_0 h z_0^{-1} = k, \quad z_0h'z_0^{-1} \equiv z_0bh'_1b^{\varepsilon_2} \ldots b^{\varepsilon_n}h'_nz_0^{-1} = bk'_1b^{\varepsilon_2} \ldots b^{\varepsilon_n}k'_n \equiv k'.$$  
Since $z_0 \neq 1$ and $k \notin F(B^*,C^*_+)$, because $z_0 h z_0^{-1} = k$ cannot hold in $F(A^*,B^*,C^*_+)$, the equality $h = z_0^{-1} kz_0$ must define an exceptional element of the intersection $F(A^*,B^*) \cap F(A^*_+,B^*,C^*)$. By Proposition 5.1 of \cite{C}, the basic exceptional relation is either of the form $u=v_0v_2$, with $z_0 \equiv v_2$ or  $u=\tilde v^{-1}v_0\tilde v$ with $z_0 \equiv \tilde v$.  The former implies that $z_0 = v_0^{-1}u \in F(B,C)F(A,B)$, which we can rule out since $z_0$ is nontrivial of type $(A^*:C^*)$, and so the latter must hold.  So we now have a counterpair of the form 
$$\tilde v h \tilde v^{-1} =k, \quad \tilde v h'\tilde v^{-1} = \tilde vbh'_1b^{\varepsilon_2} \ldots b^{\varepsilon_n}h'_n\tilde v^{-1} = bk'_1b^{\varepsilon_2} \ldots b^{\varepsilon_n}k'_n =k'.$$ 

\begin{sclm} $h'$ has uniform signature pattern.
\end{sclm}
\begin{proof} Suppose not; since the initial occurrence of $b$ in $h'$ has exponent $+1$, in the system of equalities yielded by the normal form theorem, we find, for some $j$ (which can be chosen minimal), $z_{j-1} \in U, w_j \in L$ and ${\overleftarrow {z_{j-1}}} h'_j = k'_jw_j$ so that ${\overleftarrow {z_{j-1}}} h'_jw_j^{-1} = k'_j$ defines an element of $F(A^*,B^*,C^*_-) \cap F(B^*,C^*)$ which must be exceptional since $z_{j-1} \neq 1$. Hence we can write $F(A^*,B^*,C^*_-) \cap F(B^*,C^*) = \langle p \rangle * F(B^*,C^*_-) = \langle q \rangle * F(B^*,C^*_-) $  where $p \in F(A^*,B^*,C^*_-)$ and $q \in  F(B^*,C^*)$. Corollary 5.4 of \cite{C} implies that $p = p_1p_0p_2^{-1}$, with at least one of $p_1,p_2$ nontrivial -- for otherwise $F(A^*,B^*) \cap F(B^*,C^*)$ would be exceptional, and then, by Proposition 5.2 of \cite{C}, $\tilde v$ would be trivial. 

 Now the hypotheses of Proposition 5.5 of \cite{C} are satisfied and therefore each extremal generator appears in just a single syllable of $u^{-1}\tilde v^{-1}v_0\tilde v$. Furthermore $p_2p_0^{-1}p_1q$ is a cyclic rearrangement of $u^{-1}\tilde v^{-1}v_0\tilde v$. This means that $\tilde v$ is \textit{intermediate} (ie omits all four \textit{extremal} generators $\{a_{\kappa}, a_{\lambda}, c_{\mu}, c_{\nu}\}$ -- see page 273 of  \cite{C}) and that $p_1 \equiv p_2 \equiv \tilde v$.  It follows that ${\overleftarrow {z_{j-1}}} \equiv \tilde v \equiv w_j$ 
and hence that 
\begin{eqnarray*} \tilde v h'_0bh'_1 \ldots bh'_j &=& k'_0bk'_1 \ldots bk'_j \tilde v \\  \tilde vb^{-1}h'_{j+1}b^{\varepsilon_{j+1}} \ldots b{\varepsilon_n}h'_n &=& b^{-1}k'_{j+1}b^{\varepsilon_{j+1}} \ldots b^{\varepsilon_n}k'_n \end{eqnarray*} 
yielding, of course,
\begin{eqnarray*}\tilde v h \tilde v^{-1} &=& k \\
  \tilde vh'_0bh'_1 \ldots bh'_j\tilde v^{-1} &=& k'_0bk'_1 \ldots bk'_j \\ \tilde vb^{-1}h'_{j+1}b^{\varepsilon_{j+1}} \ldots b^{\varepsilon_n}h'_n\tilde v^{-1} &=& b^{-1}k'_{j+1}b^{\varepsilon_{j+1}} \ldots b^{\varepsilon_n}k'_n. \end{eqnarray*} 
Since $h'$ is the product of $h'_0bh'_1 \ldots bh'_j$ and $b^{-1}h'_{j+1}b^{\varepsilon_{j+1}} \ldots b^{\varepsilon_n}h'_n$ and each has $b$--length less than $h'$, it follows that $h$ commutes with both and therefore with $h'$, which is a contradiction.
\end{proof}

In completing the proof of \fullref{Claim 4.1}, we can thus assume that $h'= bh'_1b\ldots bh'_n$ and $k'= bk'_1b\ldots bk'_n$. However it should be noted that, unlike in the proof of the Subclaim, we do not have any information about $F(A^*,B^*,C^*_{-}) \cap F(B^*,C^*)$ and so we do not at present know that $\tilde v$ is intermediate.

We obtain the usual system of equalities
$${\overleftarrow {\tilde v}} h'_1 = k'_1z_1, \ldots , {\overleftarrow {z_{n-1}}} h'_n = k'_n\tilde v.$$   If all these equalities hold freely, then we obtain $$h'_i=k'_i, i=1,2, \ldots,n \quad \hbox{and}
\quad {\overleftarrow {\tilde v}} \equiv z_1, {\overleftarrow {z_1}} \equiv z_2,\ldots , {\overleftarrow {z_{n-1}}} \equiv \tilde v$$ which is clearly impossible.  It follows, therefore that some equality does not hold freely and we have to analyse the sequence $$ {\overleftarrow {\tilde v}} h'_1 = k'_1z_1, {\overleftarrow {z_1}} h'_2 = k'_2z_2, \ldots {\overleftarrow {z_{n-1}}} h'_n = k'_n\tilde v.$$  
To complete the argument for Case 4.1 we require two further results which are similar in nature to Proposition 6.2 of \cite{C}.  These are stated and proved in \fullref{sec5} and will also be used below.

Since one of the inequalities does not hold freely, it follows from \fullref{Proposition 5.2} below that $\tilde v$ must be intermediate and hence \fullref{Proposition 5.5} below can be applied. If $n=1$ we have $$ {\overleftarrow {\tilde v}} h_1 = k_1\tilde v.$$ This does not hold freely and so \fullref{Proposition 5.5}(c) applies, giving $${\overleftarrow {\tilde v}} \equiv \tilde v$$ which is impossible.  More generally, pick the least $i$ such that ${\overleftarrow {z_{i-1}}} h_i = k_iz_i$, does not hold freely, with the appropriate interpretation for $i=1$ or $i=n$. Then $${\overleftarrow {\tilde v}} \equiv z_1, \quad{\overleftarrow {z_1}} \equiv z_2,\quad \ldots \quad{\overleftarrow {z_{i-2}}} \equiv z_{i-1}$$ and so $L(\overleftarrow{z_{i-1}}) = L(\tilde v)$. Again  \fullref{Proposition 5.5}(c) applies giving ${\overleftarrow {z_{i-1}}} \equiv \tilde v$ and we have another impossible situation.  This is the contradiction we require to conclude the proof of \fullref{Claim 4.1}.  \end{proof}

Unfortunately this is the point at which the argument becomes even more complicated.

Maintaining our notation $ghg^{-1} =k, gh'g^{-1}=k'$ for counterpairs,  we shall write $\rho_b(h), \rho_b(h')$ for the number of times that $b$ changes sign in reduced expressions for $h, h'$.   We essentially argue by induction on $\rho_b(h)+\rho_b(h')$.

\begin{clm}\label{Claim 4.2} There do not exist counterpairs $ghg^{-1}=k, gh'g^{-1}=k'$ satisfying $l_b(g)=0$ and min$\{l_b(h),l_b(h')\} > 0$ such that $\rho_b(h) +\rho_b(h')= 0$.
\end{clm}
\begin{proof} Suppose not; then, without loss of generality, there exists a counterpair \begin{align*}ghg^{-1} &= gh_0bh_1\ldots h_{m-1}bh_mg^{-1} = k_0bk_1 \ldots  k_{m-1}bk_m  = k \\ gh'g^{-1} &= gh'_0bh'_1\ldots h'_{n-1}bh'_ng^{-1} = k'_0bk'_1 \ldots  k'_{n-1}bk'_n  = k.\tag*{\hbox{and}}\end{align*} Among all such pairs we choose one such than $m+n =l_b(h)+l_b(h')$ is minimal.

 To begin with we have no information at all about exceptional intersections within $G^*$.  The first pair of Normal Form equalities are $gh_0=k_0z_0$ and $gh'_0=k'_0z'_0$.  Certainly $z_0,z'_0$ are both nontrivial, for otherwise $g \in F(B,C)F(A,B)$, and we can eliminate $g$ to obtain $h_0^{-1}h'_0 = z_0^{-1}k_0^{-1}k'_0z'_0$.  This equality may hold freely, for instance when $F(A^*,B^*) \cap F(A^*_+,B^*,C^*)$ is not exceptional, and then we can only have $h'_0 = h_0, k'_0 = k_0$ and $z'_0 \equiv z_0$. In this event we can then eliminate $ \overleftarrow {z_0}$ from the second pair of Normal Form equalities and analyse the resulting equality. Either we can continue to make such eliminations, successively identifying terms from the first member of the counterpair with the corresponding terms of the second or we will encounter an exceptional equality for $F(A^*,B^*) \cap F(A^*_+,B^*,C^*)$ after elimination.  If the first possibility occurs min$\{m,n\}$ times then we perform a ``Nielsen operation'' on our counterpair and contradict the minimality of $m+n$ (or obtain a counterpair with min$\{l_b(h), l_b(h')=0\}$ contradicting \fullref{Claim 4.1}). For instance if $m < n$, we obtain $$gh^{-1}h'g^{-1}= gh_m^{-1}h'_mbh'_{m+1}\ldots bh'_ng^{-1} = k_m^{-1}k'_mbk'_{m+1}\ldots bk'_n = k^{-1}k$$ (and $gh_m^{-1}h'_mg^{-1} =k_m^{-1}k'_m$ if $n=m$). A similar argument applies to the two Normal Form systems when working from the last pair back towards the first pair, only this time the elimination and identification process breaks down when we find an exceptional equality for $F(A^*,B^*,C^*_-) \cap F(B^*,C^*)$.

If both elimination and identification processes break down, then we know that both $F(A^*,B^*) \cap F(A^*_+,B^*,C^*)$ and $F(A^*,B^*,C^*_-) \cap F(B^*,C^*)$ are exceptional.  Suppose that, starting from the front, the breakdown occurs with $\smash{h_l^{-1}h'_l = z_l^{-1}k_l^{-1}k'_lz'_l}$. Then Proposition 5.2 of \cite{C} gives us a basic exceptional relator of the form $u^{-1}v_1^{-1}v_0v_2$ with, since we are free at this point to make a choice, $v_1 \equiv z_l, v_2 \equiv z'_l$. In addition, $h'_i=h_i,k'_i = k_i, z'_i=z_i, 1 \leq i \leq l-1$ and \begin{align*} gh_0bh_1\ldots h_{l-1}bh_l &= k_0bk_1\ldots k_{l-1}bk_lv_1 \\ gh_0bh_1\ldots h_{l-1}bh'_l &= k_0bk_1\ldots k_{l-1}bk'_lv_2\tag*{\hbox{and}}\end{align*}  Conjugating both equalities by $gh_0bh_1\ldots h_{l-1}bh_lv_1^{-1} = k_0bk_1\ldots k_{l-1}bk_l$ yields 
$$\displaylines{v_1bh_{l+1}b\ldots h_{m-1}bh_mh_0bh_1\ldots h_{l-1}bh_lv_1^{-1}\hfill \cr\hfill= bk_{l+1}b\ldots k_{m-1}bk_mk_0bk_1\ldots k_{l-1}bk_l ,\cr v_1h_l^{-1}h'_lbh'_{l+1}b\ldots bh'_nh_0bh_1\ldots h_{l-1}bh_lv_1^{-1}\hfill\cr \hfill= k_l^{-1}k'_lbk'_{l+1}b\ldots bk'_nk_0bk_1\ldots k_{l-1}bk_l.}$$   Relabelling, we have obtained a counterpair \begin{align*} v_1hv_1^{-1}  &=  v_1bh_1b \ldots bh_mv_1^{-1}  =  bk_1b \ldots bk_m  =  k ,\\
v_1h'v_1^{-1} &=  v_1h'_0bh'_1b \ldots bh'_nv_1^{-1}  =  k'_0bk'_1 \ldots bk'_n  =  k'.\end{align*} 
Moreover the exceptional equality for $F(A^*,B^*) \cap F(A^*_+,B^*,C^*)$ has become the initial Normal Form equality for $v_1h'v_1^{-1} =k'$ and hence we can rewrite the second equality in our counterpair as $$v_2bh'_1b \ldots bh'_nh'_0v_2^{-1} = bk'_1 \ldots bk'_nk'_0.$$
We shall use the results of  \fullref{sec5} and Proposition 6.2  of \cite{C} to derive a contradiction.  We consider three cases according as $v_1$ and $v_2$ are or are not intermediate.

\textbf{Case A}\qua Suppose neither $v_1$ nor $v_2$ is intermediate.  Then, by \fullref{Proposition 5.2},  all the equalities in the Normal Form system for $v_1bh_1b \ldots bh_mv_1^{-1} = bk_1b \ldots bk_m$ must hold freely.  In detail, then, we have 
$$\overleftarrow {v_1} \equiv z_1, \overleftarrow {z_1} \equiv z_2, \ldots, \overleftarrow {z_{m-2}} \equiv z_{m-1}, \overleftarrow {z_{m-1}} \equiv v_1$$ and hence $v_1$ is the \textit{$m$--fold downshift} $v_1(\overleftarrow m)$ of itself, which is contradictory.

\textbf{Case B}\qua Suppose both $v_1$ and $v_2$ are intermediate (and thus both $p_1$ and $p_2$ are intermediate).

Without loss of generality we can suppose that $L(v_1) = \hbox{min}\{L(v_1),L(v_2)\}$. We shall apply \fullref{Proposition 5.5} to the sequence of Normal Form Equalities derived from $$v_1bh_1b \ldots bh_mv_1^{-1} = bk_1b \ldots bk_m.$$ 
If $m=1$, there is only a single equality $\overleftarrow {v_1}h_1 = k_1v_1$. The inequality hypothesis of \fullref{Proposition 5.5} is valid, by our assumption that $L(v_1) = \hbox{min}\{L(v_1),L(v_2)\}$. The given equality cannot hold freely, since $v_1$ is distinct from $\overleftarrow {v_1}$, and, for the same reason, $\overleftarrow {v_1}h_1 = k_1v_1$ must be $v_2u^{-1}=v_0v_1$. In particular, we have $\overleftarrow {v_1} \equiv v_2$ and $L(v_2) = L(v_1) = \hbox{min}\{L(v_1),L(v_2)\}$. If we switch our attention to $v_2bh'_1b \ldots bh'_nh'_0v_2^{-1} = bk'_1b \ldots bk'_nk'_0$, then the first Normal Form equality is $\overleftarrow {v_2}h'_1 = k'_1z'_1$ and again we can apply \fullref{Proposition 5.5}. However this equality cannot be $v_2u^{-1}=v_0^{-1}v_1$, since we cannot have $\overleftarrow {v_2} \equiv v_2$, and it cannot be $v_1u=v_0v_2$, since this would give $\overleftarrow {v_2} \equiv v_1$ which is inconsistent with $\overleftarrow {v_1} \equiv v_2$.  The equality must therefore hold freely and so $z'_1 \equiv \overleftarrow {v_2}$. Now this argument iterates -- for suppose we have obtained, via free equalities, $z'_j$ as the $j$--fold downshift $v_2( \olaj)$.   Then the next Normal Form equality is $v_2( \olajpo \smash{h'_{j+1} = k'_{j+1}z'_{j+1}}$.  The possibility that this is $v_2u^{-1}=v_0v_2$ is ruled out by the fact that we cannot have $v_2( \olajpo)\equiv v_2$ while the possibility of $v_1u=v_0v_2$ is ruled out by  the inconsistency of $v_2( \olajpo) \equiv v_1$ with $\overleftarrow {v_1} \equiv v_2$.  Hence the only possibility is that $z'_{j+1} = v_2( \olajpo)$ freely.  Eventually, then we reach a final comparison $v_2( \smash{\overleftarrow n}) \equiv v_2$ and we have a contradiction.  (The iteration, of course, is unnecessary when also $n=1$.)

Now we have to dispose of the case when $m>1$. If we can show that the Normal Form equalities must all hold freely, then we have the same contradiction as in the previous case. The first equality is $\overleftarrow {v_1}h_1 = k_1z_1$; we do have $\hbox{min}\{L(\overleftarrow {v_1}),L(z_1)\} \leq L(v_1) = \hbox{min}\{L(v_1),L(v_2)\}$ and hence \fullref{Proposition 5.5} applies. If $\overleftarrow {v_1}h_1 = k_1z_1$ does not hold freely, then it is an instance of either $v_1u=v_0v_2$ or $v_2u^{-1}=v_0v_1$.  The former yields $\overleftarrow {v_1} \equiv v_1$, which is impossible, so we have to consider the latter which gives $\overleftarrow z_1 \equiv v_1$.  However in this case we have $v_1bh_1 = bk_1v_1$ and therefore $v_1bh_1b \ldots bh_mv_1^{-1} = bk_1b \ldots bk_m$ decomposes into the two equalities $v_1bh_1v_1^{-1}=bk_1$ and $v_1bh_2\ldots bh_mv_1v_1^{-1}=bk_2 \ldots bk_m$. The minimality of $m+n$ implies that both $bh_1$ and $bh_2\ldots bh_m$ commute with $h'$ whence of course so does $h$. So we can rule out the second possibility and conclude that $\overleftarrow {v_1}h_1 = k_1z_1$ holds freely.  We can now iterate this whole argument -- at each stage we can rule out both versions of case (b) of \fullref{Proposition 5.5}, the first because it implies that $v_1$ coincides with a multiple downshift of itself and the second because it implies that $v_1bh_1b \ldots bh_mv_1^{-1} = bk_1b \ldots bk_m$ decomposes.  Hence all the Normal Form equalities hold freely and as a result we deduce that $v_1$ is the same as its $m$--fold downshift $v_1( \overleftarrow m)$ and this is impossible.  This concludes the argument for Case B.

\textbf{Case C}\qua Suppose that one of $v_1$ and $v_2$ is intermediate and the other is not.

If $v_1$ is intermediate, we use the Normal Form equalities for $v_1bh_1b \ldots bh_mv_1^{-1} = bk_1b \ldots bk_m$. We can apply Proposition 6.2 of \cite{C} to the equality ${\smash{\overleftarrow {v_1}}}h_1 = k_1z_1$, since we have $L({\overleftarrow {v_1}}) = L(v_1) = d(a_{\kappa}, c_{\mu})$; see the foot of page 295 of \cite{C} where this notation is explained.  This equality either holds freely or is an instance of $v_1u=v_0v_2$.  However the latter is clearly impossible since it yields ${\overleftarrow {v_1}}\equiv v_1$ and so the equality holds freely. (When $m=1$, this is the required contradiction.) In particular $L(z_1) = L(v_1)$ and therefore $L({\overleftarrow {z_1}}) = L(v_1)$ so that Proposition 6.2 of \cite{C} also applies to ${\overleftarrow {z_1}}h_2 = k_1z_2$.  This must also hold freely since otherwise ${\overleftarrow {z_1}} \equiv v_1$ which is impossible since in fact ${\overleftarrow {z_1}} \equiv v_1(\smash{\overleftarrow 2})$.  Clearly this argument can be iterated and eventually we obtain the contradictory conclusion that $v_1(\overleftarrow m)\equiv v_1$.  When $v_2$ is intermediate, the same argument applies to the Normal Form equalities for $v_2bh'_1b \ldots bh'_nh'_0v_2^{-1} = bk'_1b \ldots bk'_nk'_0$.  This concludes the argument for \fullref{Claim 4.2}. \end{proof}

\begin{clm}\label{Claim 4.3} There do not exist counterpairs $ghg^{-1}=k, gh'g^{-1}=k'$ satisfying $l_b(g)=0$ and min$\{l_b(h),l_b(h')\} > 0$ such that $\rho_b(h) +\rho_b(h')= 1$.
\end{clm}
\begin{proof} Suppose such counterpairs exist; then without loss of generality we can assume that there is a counterpair $ghg^{-1}=k, gh'g^{-1}=k'$ satisfying
\begin{enumerate}
\item[(i)] $m+n=l_b(h)+l_b(h')$ is minimal among all counterpairs satisfying $l_b(g)=0$, min$\{l_b(h),l_b(h')\} >0$ and $\rho_b(h) + \rho_b(h') =1$;

\item[(ii)] $\rho_b(h)=0 , \rho_b(h') =1$;
 
\item[(iii)] the initial occurrences of $b$ in $h$ and $h'$ have the same exponent (since we can invert $h$ if necessary) which, without loss of generality, we can take to be $+1$.
\end{enumerate}
Suppose then that we have \begin{align*}gh_0bh_1\ldots h_{m-1}bh_mg^{-1} &= k_0bk_1 \ldots  bk_m\\
gh'_0bh'_1b\ldots bh'_jb^{-1} \ldots b^{-1}h'_ng^{-1} &= k'_0bk'_1 \ldots bk'_jb^{-1} \ldots b^{-1}k'_n,\tag*{\hbox{and}}\end{align*} satisfying (i)--(iii). We note that, since we have a change of sign from positive to negative in $h'$, $F(A^*,B^*,C^*_-)\cap F(B^*,C^*)$ is exceptional.

We have three equalities $gh_0=k_0z_0$, $gh'_0=k'_0z'_0$ and $h'_ng^{-1}={z'_n}^{-1}k'_n$ (where $z'_n$ is a shorthand for $\smash{\overrightarrow{w\smash{'}_{n-1}\strut}}$) which yield equalities 
\begin{align}
{z'_0}^{-1}{k'_0}^{-1}k_0z_0 &= {h'_0}^{-1}h_0\label{eq1}\\
{z'_n}^{-1}k'_nk_0z_0 &= h'_nh_0\label{eq2}\\
{z'_n}^{-1}k'_nk'_0z'_0 &= h'_nh'_0\label{eq3}
\end{align}
upon elimination of $g$.  Each of these either holds freely or is an exceptional equality for $F(A^*,B^*) \cap F(A^*_+,B^*,C^*)$. If all three of these equalities hold freely -- and it is easy to see that if two hold freely then so will the third -- then we have $\smash{h_0 = h'_0 = {h'_n}^{\mskip-6mu -1}}$, $\smash{k_0 = k'_0 = {k'_n}^{-1}}$, and $\smash{z_0 = z'_0 = z'_n}$.  Conjugating
 both equalities by $gh_0b=k_0z_0b = k_0b\overleftarrow {z_0}$ then yields a counterpair $z_0\hat hz_0^{-1}=\hat k, z_0\hat h'z_0^{-1}=\hat k'$. If $n=2$, then min$\smash{\{l_b(\hat h),l_b(\hat h')\} =0}$ contradicting \fullref{Claim 4.1} while if $n>2$ and $j=n-1$, then \fullref{Claim 4.2} is contradicted.  Finally if $n>2$ and $j<n-1$, then $z_0\hat hz_0^{-1}=\hat k, z_0\hat h'z_0^{-1}=\hat k'$ satisfies the same hypotheses as the original counterpair in contradiction to condition (i). It follows therefore that we can assume that two of the three equalities obtained by elimination are exceptional for $F(A^*,B^*) \cap F(A^*_+,B^*,C^*)$ and therefore of the form $u^{\pm 1}=(v_1^{-1}v_0v_2)^{\pm 1}$. Since the equality ${\overleftarrow {z_{j-1}}} h'_j = k'_jw_j$ taken from the Normal Form equalities for $\smash{ghg^{-1}=k'}$ defines an exceptional equality for $F(A^*,B^*,C^*_-) \cap F(B^*,C^*)$, we know that Proposition 5.5 of \cite{C} applies.

Suppose then, that \eqref{eq2} and \eqref{eq3} are exceptional. We shall consider other cases below after we have completed the analysis for this case. Before getting into our main argument we need, firstly, to show that in the exceptional equality $u=v_1^{-1}v_0v_2$, we have $v_1 \neq v_2$.  Since \eqref{eq2} and \eqref{eq3} are exceptional, we obtain $\{z_0, z'_n\} = \{v_1,v_2\}= \{z'_0, z'_n\}$ and hence, that $z_0=z'_0$. Suppose, by way of contradiction, that $v_1=v_2$ so that the exceptional equality is $u=v_1^{-1}v_0v_1$. Then of course $z_0=z'_0=z'_n = v_1$. On the other hand, if we conjugate our original counterpair by $gh=k_0z_0 = k_0v_1$, then we obtain a new counterpair 
\begin{align*}v_1 \hat hv_1^{-1} &= v_1bh_1\ldots h_{m-1}bh_mh_0v_1^{-1} = bk_1 \ldots  bk_mk_0 = \hat k\\
v_1\hat h'v_1^{-1} &= v_1h_0^{-1}h'_0b\ldots bh'_jb^{-1} \ldots  b^{-1}h'_nh_0v_1^{-1} = k_0^{-1}k'_0b\ldots  bk'_jb^{-1} \ldots b^{-1}k'_nk_0\\ &\hphantom{= v_1h_0^{-1}h'_0b\ldots bh'_jb^{-1} \ldots  b^{-1}h'_nh_0v_1^{-1}\,}=\hat k'.\end{align*}  However it follows from \fullref{Claim 4.1} applied to the pair $$v_1uv_1^{-1} = v_0, \  v_1bh_1\ldots h_{m-1}bh_mh_0v_1^{-1} = bk_1 \ldots  bk_mk_0$$ that $u$ commutes with $\hat h$ and similarly that $u$ commutes with $\hat h'$ in $F(A^*,B^*)$. Since this means that $\hat h$ and $\hat h'$ commute we have the contradiction needed to ensure that $v_1 \neq v_2$.

We still have $\{z_0, z'_n\} = \{v_1,v_2\}= \{z'_0, z'_n\}$ and $z_0=z'_0$. By exercising a choice for our notation we can assume that $z_0=z'_0= v_1$ so that we can deduce from \eqref{eq2} and \eqref{eq3} that $k'_nk_0= v_0^{-1}=k'_nk'_0$ and $h'_nh_0 =u^{-1} = h'_nh'_0$.  This implies that $h_0=h'_0$ and $k_0=k'_0$ and so, conjugating our original counterpair by $gh_0 = k_0z_0 =k_0v_1$ as before we obtain \begin{align*}v_1 \hat hv_1^{-1} &= v_1bh_1\ldots h_{m-1}bh_mh_0v_1^{-1} = bk_1 \ldots  bk_mk_0 = \hat k\\
v_1\hat h'v_1^{-1}  &=  v_1h_0^{-1}h'_0bh'_1b\ldots bh'_jb^{-1} \ldots b^{-1}h'_nh_0v_1^{-1} \tag*{\hbox{and}}\\        &= k_0^{-1}k'_0bk'_1 \ldots bk'_jb^{-1} \ldots b^{-1}k'_nk_0  =  \hat k'.\end{align*}
Again we have to break the analysis down into three separate cases, depending on the properties of $v_1$ and $v_2$.

{\bf Case A}\qua Suppose that neither $v_1$ nor $v_2$ is intermediate.  Then we can apply \fullref{Proposition 5.2} to $v_1bh_1\ldots bh_mh_0v_1^{-1} = bk_1\ldots bk_mk_0$ to obtain the contradiction $v_1(\overleftarrow m) \equiv v_1$.

{\bf Case B}\qua Suppose that both $v_1$ and $v_2$ are intermediate.  Then both $p_1$ and $p_2$ are intermediate and $p_1=v_1, p_2=v_2$.  This means that the equality involving the change of sign decomposes at the change of sign and it follows then from \fullref{Claim 4.2} that $bh_1\ldots bh_mh_0$ commutes with the two constituent factors of $bh'_1b\ldots bh'_jb^{-1} \ldots h'_{n-1}b^{-1}h'_nh_0$ and hence with $\smash{bh'_1b\ldots bh'_jb^{-1} \ldots h'_{n-1}b^{-1}h'_nh_0}$ itself, which rules out this case.

{\bf Case C}\qua Suppose that one of $v_1$ or $v_2$ is intermediate and the other is not.  If $v_1$ is intermediate, then we can apply Proposition 6.2 of \cite{C} to the uniform signature equality as we did in Case C of \fullref{Claim 4.2}.  The problem case is when $v_2$ is intermediate and $v_1$ is not.

We have $\overleftarrow{v_1}h_1 = k_1z_1$ and $\overleftarrow{v_1}h'_1 = k_1z'_1$ (where temporarily we assume that $j \ge 2$).  If the resulting equality is exceptional for $F(A^*,B^*) \cap F(A^*_+,B^*,C^*)$, then $\{z_1,z'_1\}=\{v_1,v_2\}$ and so $v_1bh_1=bk_1v_1$ whence $v_1bh_1v_1^{-1}=bk_1$ or $v_1bh_1=bk_1v_2=bk_1v_0^{-1}v_1u$ whence $v_1bh_1u^{-1}v_1^{-1} =bk_1v_0^{-1}$.  In either case we can decompose $\smash{v_1bh_1\ldots bh_mh_0v_1^{-1} = bk_1\ldots bk_mk_0}$. Since each of the two factors of $bh_1\ldots bh_mh_0$ contains fewer than $m$ occurrences of $b$ we can use the minimality of $m+n$ in (i) to deduce that each factor of $bh_1\ldots bh_mh_0$, and hence $bh_1\ldots bh_mh_0$ itself, will commute with $bh'_1b\ldots bh'_jb^{-1} \ldots h'_{n-1}b^{-1}h'_nh_0$, again giving us a contradiction.  We can iterate this argument, either until we exhaust the uniform signature pattern equality (when $m \leq j$) or until we have obtained $ h_1=h'_1, \ldots h_{j-1} = h'_{j-1}, k_1=k'_1, \ldots k_{j-1} = k'_{j-1}$ and $z_1=z'_1, \ldots z_{j-1} = z'_{j-1}$ (when $j<m$). The former means that we can apply a Nielsen move to reduce $m+n$ -- and so can be ruled out -- while the latter means that the change of sign equality in the second term is $\smash{\overleftarrow {z_{j-1}}h'_j = k'_jw'_j}$ and we can perform an elimination with $\overleftarrow {z_{j-1}}h_j = k_jz_j$, and we have this step immediately if $j=1$. This yields $z_j^{-1}k_j^{-1}k'_jw'_j = h_j^{-1}h'_j$

We claim that $w'_j$ is intermediate.  Since $\{\overleftarrow{z_{j-1}},w'_j\} = \{p_1,p_2\}$, and either $p_1$ or $p_2$ is intermediate, we have to rule out the possibility that $\smash{\overleftarrow{z_{j-1}}}$ is intermediate.  So suppose it is intermediate and, without loss of generality, suppose that $\overleftarrow{z_{j-1}}=p_1$. Then of course $p_2$ is not intermediate. Then $L(z_{j-1}) = d(a_{\lambda},c_{\nu})$ and we can apply Proposition 6.2 of \cite{C} to $\overleftarrow {z_{j-2}}h_{j-1} = k_{j-1}z_{j-1}$.  The possible outcomes are that the equality holds freely or that it is an instance of either $p_1p_0=qp_2$ or $\smash{p_2p_0^{-1}=q^{-1}p_1}$.  The latter equality would give $\smash{p_1 \equiv z_{j-1} \equiv \overrightarrow {p_1}}$, which is impossible, and the former equality would give $p_2 \equiv \overrightarrow {p_1}$, contradicting the fact that $p_2$ is not intermediate.  Thus we are left with the outcome that $h_{j-1}=k_{j-1} = 1$ and $\smash{\overleftarrow {z_{j-2}} \equiv  z_{j-1} \equiv p_1(\overrightarrow 2)}$.  Then, however, $L(z_{j-2}) = d(a_{\lambda},c_{\nu})$ and we can clearly iterate.  We finish up with $v_1 \equiv p_1(\oraj)$, contradicting the fact that $v_1$ is not intermediate, and so $w_j'$ is intermediate.

Given that $w'_j$ is intermediate, the equality $z_j^{-1}k_j^{-1}k'_jw'_j = h_j^{-1}h'_j$ either holds freely or is exceptional for $F(A^*,B^*) \cap F(A^*_+,B^*,C^*)$ with $\{z_j, w'_j\} = \{v_1,v_2\}$.  If the latter holds, then we can again decompose the uniform signature term of our counterpair, leading to a contradiction, so we can conclude that the equality holds freely and $h_j = h'_j, k_j = k'_j$ and $w'_j \equiv z_j$.  Our counterpair can then be broken down into the three equalities
\begin{gather*}v_1bh_1 \ldots bh_j = bk_1 \ldots bk_iz_j, \quad z_jbh_{j+1} \ldots bh_mh_0v_1^{-1} = bk_{j+1} \ldots bk_mk_0v_1^{-1}
\\z_ib^{-1}h'_{j+1}\ldots b^{-1}h'_nh_0v_1^{-1} = b^{-1}k'_{j+1}\ldots b^{-1}k'_nk_0.\tag*{\hbox{and}}\end{gather*}
If we ``splice'' the second and third equalities together, after inverting the former, then we obtain 
\begin{multline*}v_1h_0^{-1}{h'_n}^{\mskip-6mu -1}b{h'_{n-1}}^{\mskip-6mu -1} \ldots {h'_{j+1}}^{\mskip-6mu -1}bbh_{j+1} \ldots bh_mh_0v_1^{-1}\\= k_0^{-1}{k'_n}^{-1}b{k'_{n-1}}^{-1} \ldots {k'_{j+1}}^{-1}bbk_{j+1} \ldots bk_mk_0.\end{multline*}
If we combine this with $$v_1bh_1\ldots bh_mh_0v_1^{-1} = bk_1\ldots bk_mk_0,$$ then we have two equalities each with uniform signature pattern. So we can apply \fullref{Claim 4.1} to obtain the commuting relation 
\begin{multline*}h_0^{-1}{h'_n}^{\mskip-6mu -1}b{h'_{n-1}}^{\mskip-6mu -1} \ldots {h'_{j+1}}^{\mskip-6mu -1}bbh_{j+1} \ldots bh_mh_0bh_1\ldots bh_jbh_{j+1} \ldots bh_mh_0 \\ =
bh_mh_0bh_1\ldots bh_jbh_{j+1} \ldots bh_mh_0h_0^{-1}{h'_n}^{\mskip-6mu -1}b{h'_{n-1}}^{\mskip-6mu -1} \ldots {h'_{j+1}}^{\mskip-6mu -1}bbh_{j+1} \ldots bh_mh_0.\end{multline*} 
We can cancel $bh_{j+1} \ldots bh_mh_0$ to yield 
\begin{multline*}h_0^{-1}{h'_n}^{\mskip-6mu -1}b{h'_{n-1}}^{\mskip-6mu -1} \ldots {h'_{j+1}}^{\mskip-6mu -1}bbh_{j+1} \ldots bh_mh_0bh_1\ldots bh_j \\=  bh_mh_0bh_1\ldots bh_jbh_{j+1} \ldots bh_mh_0h_0^{-1}{h'_n}^{\mskip-6mu -1}b{h'_{n-1}}^{\mskip-6mu -1} \ldots {h'_{j+1}}^{\mskip-6mu -1}b.\end{multline*}
The above equalities hold in $F(A,B)$ which is, however, as a subgroup of $$G = \langle G^*, b \ | \  bF(A^*_-,B^*,C^*_-)b^{-1} = F(A^*_+,B^*,C^*_+) \rangle,$$ expressed as the HNN-extension $\langle A^*,B^*, b \ | \ bF(A^*_-,B^*)b^{-1} = F(A^*_+,B^*) \rangle$ in this context. Both expressions in the equalities are reduced and hence by the Normal Form Theorem applied to the last pair of occurrences of $b$ in the second of the two equalities, we deduce that $h_j \in F(A^*_-,B^*)$.  However that fact that the change of sign term of our counterpair is given in reduced form means that $h_j = h'_j \notin F(A^*_-,B^*)$.  This contradiction completes our analysis of the case when the equalities \eqref{eq2} and \eqref{eq3} are exceptional.

This leaves us with the remaining two possibilities for whichever pair of \eqref{eq1}, \eqref{eq2}, \eqref{eq3} are exceptional.  If \eqref{eq3} and \eqref{eq1} are exceptional so that $z_0=z'_n$, we can apply the above analysis to $ghg^{-1}=k$ and $g{h'}^{-1}g^{-1}={k'}^{-1}$, with $z'_n$ in the role of $z'_0$, to deduce the desired contradiction immediately. 

On the other hand, if \eqref{eq1} and \eqref{eq2} hold then we cannot deduce our conclusion by the same kind of appeal to symmetry since what we know this time from the analogue of the initial steps of our analysis above is that $z'_0=z'_n$ and this does not provide a connection between the two terms of our counterpair but rather a connection between the two ends of the term that contains a sign change.  The result is that when we carry out further stages of the analysis, what we obtain, after choosing our notation so that $z_0 = v_1$ and $z'_0=z'_n = v_2$, is the pair of equalities \begin{align*}v_1bh_1\ldots bh_mh_0v_1^{-1} &= bk_1\ldots bk_mk_0\\
v_2bh'_1b\ldots bh'_jb^{-1} \ldots h'_{n-1}b^{-1}v_2^{-1} &= bk'_1 \ldots bk'_jb^{-1} \ldots k'_{n-1}b^{-1}\tag*{\hbox{and}}\end{align*} (which strictly speaking do not form a counterpair since $v_1 \neq v_2$).

We can, however, dispose of the Cases A and B for $v_1$ and $v_2$, ie neither is or both are intermediate exactly as we did in the previous case. So again the difficult case is when just one is intermediate, and in fact the case when $v_2$ is intermediate is the problem (since if $v_1$ is intermediate we can ``attack'' the Normal Form sequence for  $v_1bh_1\ldots bh_mh_0v_1^{-1} = bk_1\ldots bk_mk_0$ with Proposition 6.2 of \cite{C} as we did in the case when \eqref{eq2} and \eqref{eq3} were exceptional).

The trick is to attack the Normal Form equalities for $$v_2bh'_1b\ldots bh'_jb^{-1} \ldots h'_{n-1}b^{-1}v_2^{-1} = bk'_1 \ldots bk'_jb^{-1} \ldots k'_{n-1}b^{-1}$$ from both ends simultaneously.  (This is the analogue of attacking the two Normal Form equalities from one end.) The first and last terms are $\overleftarrow {v_2}h'_1 =k'_1z'_1$ and $h'_{n-1}\overleftarrow {v_2}^{-1} = \smash{{{z'_{n-1}}}^{\mskip -6mu -1}k'_{n-1}}$ (adapting our notation suitably and temporarily assuming that $1 < j < n-1$).  We can eliminate $\overleftarrow {v_2}$ and the result is ${z'_{n-1}}^{-1}k'_{n-1}k'_0z'_0 = h'_{n-1}h'_1$.  If this is exceptional then we can decompose $$v_2bh'_1b\ldots bh'_jb^{-1} \ldots h'_{n-1}b^{-1}v_2^{-1} = bk'_1 \ldots bk'_jb^{-1} \ldots k'_{n-1}b^{-1}$$ and, in the usual manner, obtain a contradiction.  So the equality must hold freely and we obtain $h'_{n-1}= {h'_0}^{-1}$.  This argument will iterate and hence, taking inverses if necessary to ensure that $j \ge n-j$ we eventually reach a point where we can rewrite our equality as
$$v_2bh'_1b\ldots bh'_jb^{-1}{h'_{l-1}}^{\mskip-6mu -1} \ldots {h'_1}^{-1}b^{-1}v_2^{-1} = bk'_1 \ldots bk'_jb^{-1}{k'_{l-1}}^{\mskip-6mu -1} \ldots {k'_1}^{-1}b^{-1}$$
where $l \leq j$.  Since $v_2$ is intermediate we can apply Proposition 6.2 of \cite{C} to the string of equalities $\overleftarrow {v_2}h'_1 =k'_1z'_1, \olazp{1}h'_2 =k'_2z'_2, \ldots, \olazp{j-2}h'_{j-1} =k'_{j-1}z'_{j-1}$ and deduce that $v_2(\olaj) \equiv \olazp{j-1}$.  Moreover, since $l \leq j$ we also obtain the equality $v_2(\smash{\overleftarrow l}) \equiv \olazp{l-1}$.  

Still assuming that $1<j<n-1$, we deduce that $\{\olazp{j-1}, \olazp{l-1}\} = \{p_1,p_2\}$, but this is also true when $j=1$ with $\overleftarrow  {v_2}$ in place $\olazp{j-1}$, or $j=n-1$ with $\overleftarrow  {v_2}$ in place $\olazp{l-1}$. However, since $v_2$ is intermediate, it follows that $v_2$ is a proper subword of whichever of $p_1,p_2$ is not intermediate. Therefore, for instance if $u=v_{13}^{-1}p_0^{-1}v_{12}^{-1}qv_{11}^{-1}v_0v_2$, then $p_1\equiv v_{12}$ and $p_2 \equiv v_{11}^{-1}v_0v_2u^{-1}v_{13}^{-1}$ and this is impossible.  This completes the proof of \fullref{Claim 4.3}. \end{proof}

\begin{clm}\label{Claim 4.4} There do not exist counterpairs $ghg^{-1} =k, gh'g^{-1}=k'$ satisfying $l_b(g)=0$ and {\rm min}$\{l_b(h),l_b(h')\} >0$, such that $\rho_b(h)=1= \rho_b(h')$ and $h$ and $h'$ have the same exponent on the respective initial occurrences of $b$. \end{clm}
\begin{proof} Suppose not; then, without loss of generality, we have a counterpair of the form
\begin{align*}gh_0bh_1 \ldots bh_ib^{-1}h_{i+1} \ldots b^{-1}h_mg^{-1} &= k_0bk_1 \ldots bk_ib^{-1}k_{i+1} \ldots b^{-1}k_m\\
gh_0bh'_1 \ldots bh'_jb^{-1}h'_{j+1} \ldots b^{-1}h'_ng^{-1} &= k'_0bk'_1 \ldots bk'_jb^{-1}k'_{j+1} \ldots b^{-1}k'_n.\tag*{\hbox{and}}\end{align*} and we can assume that we have chosen this counterpair with $m+n$ minimal among all possible candidates.

 We observe firstly that the changes of sign from positive to negative show that the intersection $F(A^*,B^*,C^*_-) \cap F(B^*,C^*)$ is exceptional.  By taking inverses if necessary, we can then assume that $\overleftarrow{z_{i-1}}h_iw_i^{-1} = \olazp{j-1} h'_j{w'_j}^{-1} = ((p_1p_0p_2^{-1})^{\pm 1}$ and, in particular, that $\smash{z_{i-1} = z'_{j-1}, w_i = w'_j}$ and $\smash{h_i = h'_j =p_0^{\pm 1}, k_i= k'_j = q^{\pm 1}}$.

 By eliminating $g$ variously from the equalities $gh_0 = k_0z_0$, $\smash{gh'_0 = k'_0z'_0}$, $h_mg^{-1} = \smash{z_m^{-1}}k_m$, $h'_ng^{-1} = {z'_n}^{-1}k'_n$, we obtain the following six equalities:
 
\begin{enumerate} \item [(1)] ${z'_0}^{-1}{k'_0}^{-1}k_0z_0 = {h'_0}^{-1}h_0$
 
\item [(2)] $z_m^{-1}k_mk_0z_0 = h_mh_0$

\item [(3)] ${z'_n}^{-1}k'_nk_0z_0 = h'_nh_0$

\item [(4)] $z_m^{-1}k_mk'_0z'_0 = h_mh'_0$

\item [(5)] ${z'_n}^{-1}k_mk'_0z'_0 = h_mh'_0$

\item [(6)] ${z'_n}^{-1}k'_nk_m^{-1}z_m = h'_nh_m^{-1}$
\end{enumerate}

In general each of these will either hold freely or be an exceptional equality for $F(A^*,B^*) \cap F(A^*_+,B^*,C^*)$.  We need to know exactly what the possibilities are. This is most easily done as a separate lemma within the current argument.

\begin{llem} Let $h_i, i=1,2,3,4$ be nontrivial elements of $F(A^*,B^*)$, $k_i, i=1,2,3,4$ nontrivial elements of $F(B^*,C^*)$ and $z_i, i=1,2,3,4$ nontrivial elements of $U$ such that the six equalities $z_i^{-1}k_i^{-1}k_jz_j, 1 \leq i,j, \leq 4, i \neq j$ hold.  Then 
\begin{enumerate} \item [\rm(i)] either there exists $i$ such that all the equalities involving $z_i$ hold freely in which case all six equalities hold freely and hence coincide;
\item [\rm(ii)] or there exists a partition of $\{1,2,3,4\}$ into subsets $\{i,j\}$ and $\{i',j'\}$ such that $z_i^{-1}k_i^{-1}k_jz_j = h_i^{-1}h_jz$ and $z_{i'}^{-1}k_{i'}^{-1}k'_jz'_j= h_{i'}^{-1}h'_j$ hold freely and the remaining equalities are all exceptional equalities for $F(A^*,B^*) \cap F(A^*_+,B^*,C^*)$ and therefore coincide up to possible inversion.
\end{enumerate}
\end{llem}
\begin{proof}  It is convenient to visualise the equalities as the edges of a tetrahedron whose vertices are the elements $z_i, i=1,2,3,4$.  It is easy to see that if the equalities on two edges of a face are free, the so is the equality on the third edge.  It follows that if there exists $i$ such that all three edges incident to the vertex $z_i$ represent free equalities, then all six equalities hold freely and therefore (i) holds.

Suppose then that every vertex $z_i$ is incident to at most one edge that is free, ie represents a free equality. We need to show that then (ii) holds.  For this we need the following observation.

\begin{slem}  If two of the equalities involving the element $z_i$ are exceptional, then the equality obtained by eliminating $z_i$ from these holds freely. 
\end{slem} 
\begin{proof} Suppose, without loss of generality, that the equalities $\smash{z_1^{-1}k_1^{-1}k_2z_2 = h_1^{-1}h_2}$ and $\smash{z_1^{-1}k_1^{-1}k_3z_3 = h_1^{-1}h_3}$ are both exceptional.  Then both are an instance of the equality $\smash{v_1^{-1}v_0v_2}=u$ (or its inverse) and we have, say, $z_1=v_1, z_2=z_3=v_2$ and $k_1^{-1}k_2 = v_0 = k_1^{-1}k_3$ so that $k_2=k_3$ and similarly $h_2=h_3$.  Then clearly $z_2^{-1}k_2^{-1}k_3z_3 = h_2^{-1}h_3$ holds freely. \end{proof}

Suppose then that, say, the edges $z_1z_2$ and $z_1z_3$ are exceptional, ie represent exceptional equalities. By the Sublemma, the third edge $z_2z_3$ of the face $z_1z_2z_3$ is free.  Since at most one edge incident to $z_1$ can be free,  it follows that $z_1z_4$ is exceptional and similarly $z_3z_4$ is exceptional.  By the Sublemma, $z_2z_4$ is free and we have the partition consisting of $\{1,3\}$ and $\{(2,4\}$ as required. Finally we note that if a face has two edges that are exceptional, then using the free inequality on the third edge transforms the exceptional equality on one edge 
into the exceptional equality on the other. \end{proof}

We return to the argument of \fullref{Claim 4.4}.  
If all the equalities obtained by substituting for $g$ hold freely, then by conjugating by $gh_0 =kz_0$, we can obtain a conjugate counterpair but at the same time reduce both $m$ and $n $ by $2$.  This will contradict the minimality of our choice of counterpair, although care must be taken in ``degenerate'' cases when one of our conditions {\rm min}$\{l_b(h),l_b(h')\} >0$ or $\rho_b(h)=1= \rho_b(h')$ fails to hold for the new counterpair.  However these ``degenerate'' cases can all be dealt with by appealing to our earlier results Claims \ref{Claim 4.1}--\ref{Claim 4.3}. Therefore we only have to deal with the case when we have four exceptional and two free equalities. We encounter the same three cases as in \fullref{Claim 4.3}, depending on the nature of $v_1$ and $v_2$ in the exceptional equality $u=v_1^{-1}v_0v_2$ for $F(A^*,B^*) \cap F(A^*_+,B^*,C^*)$.

{\bf Case A}\qua Suppose neither $v_1$ nor $v_2$ is intermediate.  

In this situation, it follows from \fullref{Proposition 5.2} that all of the Normal Form equalities other than the first, last and ``change of sign'' term of each sequence will hold freely and thus our counterpair takes the form 
\begin{align*}gh_0b^ih_ib^{-(m-i)}h_mg^{-1} &= k_0b^ik_ib^{-(m-i)}k_m \\ gh'_0b^jh'_jb^{-(n-j)}h'_ng^{-1} &= k'_0b^jk'_jb^{-(n-j)}k'_n.\tag*{\hbox{and}}\end{align*} In particular we have $\smash{gh_0b^ih_i = k_0b^ik_ip_{\delta}}$ and $\smash{gh'_0b^jh'_j = k'_0b^jk'_jp_{\delta}}$, where $\delta = 1$ if $h_i = h'_j =p_0$ and $\delta = 2$ if $\smash{h_i = h'_j =p_0^{-1}}$.  From this we obtain $\smash{gh_0b^{i-j}{h'_0}^{-1}g^{-1}} = \smash{k_0b^{i-j}{k'_0}^{-1}}$. We claim that in fact $i=j$.  If not, then $\smash{h_0b^{i-j}{h'_0}^{-1}}$ and $\smash{k_0b^{i-j}{k'_0}^{-1}}$ have nonzero $b$--length and no sign changes.  However it follows from \fullref{Claim 4.3} that $gh_0b^{i-j}{h'_0}^{-1}g^{-1} = k_0b^{i-j}{k'_0}^{-1}$ does not form a counterpair with either of the terms of our original counterpair.  In particular this means that the nontrivial element $h_0b^{i-j}{h'_0}^{-1}$ commutes with both $h$ and $h'$.  However these commuting relations hold in the free group $F(A,B)$ and hence $h$ and $h'$ commute which of course is a contradiction.  It follows, therefore that $i=j$. A similar argument derived from $p_{\delta}b^{-(m-i)}h_m= b^{-(m-i)}k_m$ and $p_{\delta}b^{-(n-j)}h'_m= b^{-(n-j)}k'_m$ shows that $m-i =n-j$ and hence, since $i=j$, we obtain $m=n$.

However since $i=j$ and therefore $z_{i-1} = z'_{i-1}$, the Normal Form equalities $$gh_0=k_0z_0, \overleftarrow {z_1} =z_2, \ldots, \overleftarrow {z_{i-2}} =z_{i-1}\quad\text{and}\quad gh'_0=k'_0z'_0, \olazp{1} =z'_2, \ldots \olazp{i-2} =z'_{i-1}$$ yield $z_{i-2} = z'_{i-2}, \ldots, z_0 = z'_0$. Similarly $w_i = w'_i$ yields $w_{i+1} = w'_{i+1}, \ldots, w_{m-1} = w'_{m-1}$. Therefore $\smash{z_m^{-1}} = \overrightarrow {w_{m-1}} = \smash{\overrightarrow {w\smash{\rlap{$'$}}_{m-1}\strut}} = \smash{{z'_m}^{-1}}$ so that $\smash{z_m = z'_m}$.  This means that both the equalities $\smash{{z'_0}^{-1}{k'_0}^{-1}k_0z_0 = {h'_0}^{-1}h_0}$ and $\smash{{z'_m}^{-1}{k'_m}k_mz_m = {h'_m}^{-1}h_m}$ derived by elimination of $g$ from the first and last terms of the two Normal Form equalities must hold freely -- for otherwise we would have $v_1=v_2$ which, by the single syllable criterion of Proposition 5.5 of \cite{C} would mean that $v_1=v_2$ would be intermediate.  It follows therefore that $h_0=h'_0, k_0=k'_0, h_m = h'_m, k_0=k'_m$ and hence that $h=h', k=k'$ which is obviously a contradiction. This concludes Case A.

{\bf Case B}\qua Suppose both $v_1$ and $v_2$ are intermediate.
  
Our conventions on the choice of notation described after Proposition 5.5 of \cite{C} imply that $v_1 = p_1$, $v_2=p_2$.   Since we know that precisely four of the inequalities obtained by eliminating $g$ are exceptional, it follows that all of the four ``auxiliary terms'' $z_0, z'_0, z_m, z'_n$ are either $v_1$ or $v_2$.  Suppose, for instance that $z_0 \equiv v_1$.  Then we obtain a conjugate counterpair of the form
\begin{align*}z_0bh_1 \ldots bh_ib^{-1}\ldots b^{-1}h_mh_0z_0^{-1} &= bk_1 \ldots bk_ib^{-1}\ldots b^{-1}k_mk_0\\
\qquad z_0h_0^{-1}h'_0bh'_1 \ldots bh'_jb^{-1}\ldots b^{-1}h'_nh_0z_0^{-1} &= k_0^{-1}k'_0bk'_1 \ldots bk'_jb^{-1}\ldots b^{-1}k'_nk_0.\tag*{\hbox{\rlap{and}}}\end{align*}
The ``change of sign'' equalities are, as in Case A, $\smash{\overleftarrow{z_{i-1}}h_iw_i^{-1}} = \olazp{j-1} \smash{h'_j{w'_j}^{-1}} = \smash{(p_1p_0p_2^{-1})^{\pm 1}}$ and it follows that $\{\overleftarrow{z_{i-1}},w_i\} = \{\olazp{j-1},w'_j\} = \{p_1,p_2\} = \{v_1, v_2\}$. Possibly by inverting one or both of the elements of this counterpair, we can assume that $w_i=w'_j = p_1=v_1$. Then each of the displayed equalities in the above counterpair decomposes into a product of equalities with uniform signature patterns and the desired contradiction will follow from \fullref{Claim 4.2}.

{\bf Case C}\qua Suppose that one of $v_1$ and $v_2$ is intermediate and the other is not.

As in Case B, we know that from the analysis of the equalities obtained by eliminating $g$ from the first and last terms of the Normal Form inequalities, that each of $z_0, z'_0, z_m, z'_n$ is either $v_1$ or $v_2$. Again the ``change of sign'' inequalities yield $\{\overleftarrow{z_{i-1}},w_i\} = \{\olazp{j-1},w'_j\} = \{p_1,p_2\}$. We note that one of $\{p_1,p_2\}$ is intermediate and the other is not.

 Let us assume that $\overleftarrow{z_{i-1}} \equiv \olazp{j-1} \equiv p_1$ is intermediate. We shall see that there is no loss of generality in so doing.  We examine the two sequences of Normal Form equalities as far as the change of sign equalities.  Since $\overleftarrow{z_{i-1}} \equiv p_1$, and the latter is intermediate, it follows that $L(z_{i-1}) = L(p_1)=d(a_{\lambda},c_{\nu})$ and hence, by Proposition 6.2 of \cite{C}, the equality $\overleftarrow{z_{i-2}}h_{i-1} = k_{i-1}z_{i-1}$ either holds freely or is an instance of $p_2p_0^{-1}=q^{-1}p_1$.  The latter is impossible since then $z_{i-1} \equiv p_1 \equiv \overleftarrow{z_{i-1}}$.  Thus the equality holds freely and $z_{i-2} \equiv p_1(\overrightarrow 2)$.  This argument can clearly now be iterated to obtain $z_0 \equiv p_1(\orai)$ and $h_1 = \ldots = h_{i-1} = 1 = k_1 \ldots = k_{i-1}$. Applying this whole argument to the second term of our counterpair yields $z'_0 \equiv p_1(\oraj)$ and $h'_1 = \ldots = h'_{j-1} = 1 = k'_1 \ldots = k_{'-1}$.  

 We consider the equality $z_0^{-1}k_0^{-1}k'_0z'_0=h_0^{-1}h'_0$ obtained by the elimination of $g$.  If this is not free then $\{z_0, z'_0\} = \{v_1,v_2\}$. Since $z_0 \equiv p_1(\orai)$ and $z'_0 \equiv p_1(\oraj)$, we have $L(z_0)=L(p_1)=L(\smash{z'_0})$. However it follows from the relationship between $u^{-1}v_1^{-1}v_0v_2$ and $p_0^{-1}p_1^{-1}qp_2$ determined by the single syllable criterion of Proposition 5.5 of \cite{C} that $L(p_1) < L(v_1)$ or $L(p_1) < L(v_2)$ according as $v_1$ or $v_2$ is not intermediate.  Therefore $z_0^{-1}k_0^{-1}k'_0z'_0=h_0^{-1}h'_0$ can only hold freely so that $z_0=z'_0$, giving $i=j$, and also $h_0=h'_0, k_0=k'_0$.

 Our equalities therefore simplify to
\begin{align*}gh_0b^ip_0b^{-1}h_{i+1} \ldots b^{-1}h_mg^{-1} &= k_0b^iqb^{-1}k_{i+1} \ldots b^{-1}k_m\\
gh_0b^ip_0b^{-1}h'_{i+1} \ldots b^{-1}h'_ng^{-1} &= k_0b^iqb^{-1}k'_{i+1} \ldots b^{-1}k'_n.\tag*{\hbox{and}}\end{align*}
and we have just three distinct Normal Form equalities that involve $g$, namely $gh_0=k_0z_0, h_mg^{-1}=z_m^{-1}k_m$ and $h'_ng^{-1}={z'_n}^{-1}k'_n$.  These give rise to three derived equalities by elimination of $g$, namely

\begin{enumerate} \item [(2)] $z_m^{-1}k_mk_0z_0 = h_mh_0$

 \item [(4)] ${z'_n}^{-1}k'_nk_0z_0 = h'_nh_0$

 \item [(6)] ${z'_n}^{-1}k'_nk_m^{-1}z_m = h'_nh_m^{-1}$, 
\end{enumerate}
 using our earlier numbering.

 If we conjugate by $gh_0b^i = k_0b^ip_1$,and use the fact that $\overleftarrow{z_{i-1}}h_i = k_iw_i$ is just $p_1p_0=qp_2$, we obtain a conjugate counterpair
\begin{align*}p_2b^{-1}h_{i+1} \ldots b^{-1}h_mh_0b^ip_0p_2^{-1} = b^{-1}k_{i+1} \ldots b^{-1}k_mk_0b^iq\\
p_2b^{-1}h'_{i+1} \ldots b^{-1}h'_nh_0b^ip_0p_2^{-1} = b^{-1}k'_{i+1} \ldots b^{-1}k'_nk_0b^iq.\tag*{\hbox{and}}\end{align*}
 Now if either $h_mh_0 = 1 = k_mk_0$ or $h'_nh_0=1=k'_nk_0$, then we will contradict the minimality of our initial choice of counterpair. The fact that the initial exponent is now $-1$ rather than $+1$ is not an issue since our choice of $+1$ was without loss of generality and made only for notational simplicity.  However a caveat concerning the need to apply Claims \ref{Claim 4.1}--\ref{Claim 4.3} to dispose of ``degenerate'' cases does apply here as well.  This means neither (2) nor (4) can hold freely and therefore (6) will hold freely yielding $z_m=z'_n$ and $h_m=h'_m, k_m=k'_n$.

 We can now simplify our original counterpair a little further to give 
\begin{align*}gh_0b^ip_0b^{-1}h_{i+1} \ldots b^{-1}h_mg^{-1} &= k_0b^iqb^{-1}k_{i+1} \ldots b^{-1}k_m\\
gh_0b^ip_0b^{-1}h'_{i+1} \ldots h'_{n-1}b^{-1}h_mg^{-1} &= k_0b^iqb^{-1}k'_{i+1} \ldots k'_{n-1}b^{-1}k_m.\tag*{\hbox{and}}\end{align*}
and since we know that $z_m=z'_n$ we can attack the terms of our counterpair from the back via the Normal Form equalities.  Specifically we obtain $\smash{h_{m-1}}\olazmi=\smash{\smash{z_{m-1}^{-1}}k_{m-1}}$ and $h'_{m-1}\olazmi={z'_{m-1}}^{-1}k'_{m-1}$, and hence $h'_{n-1}h_{m-1}^{-1} = {z'_{n-1}}^{-1}k'_{n-1}k_{m-1}^{-1}z_{m-1}$. 

Suppose this is exceptional for $F(A^*,B^*) \cap F(A^*_+,B^*,C^*)$. Then $\{z'_{n-1}, z_{m-1}\} = \{v_1,v_2\} = \{z_0,z_m\}$. We write $\{v_1,v_2\} = \{v_{\gamma},v_{\delta}\}$ where $v_{\gamma}$ is intermediate and $v_{\delta}$ is not.  Now it follows from the relationship between $u^{-1}v_1^{-1}v_0v_2$ and $p_0^{-1}p_1^{-1}qp_2$ defined by Proposition 5.5 of \cite{C} that $L(v_{\gamma})+L(v_{\delta}) = L(v_1)+L(v_2)= L(p_1)+L(p_2)$ and also that $L(p_1) < L(v_{\delta})$ and $L(v_{\gamma}) < L(p_2)$.  Since $z_0 = p_1(\orai)$ we have $L(z_0)=L(p_1)< L(v_{\delta})$ and therefore $z_0= v_{\gamma}, z_m=v_{\delta}$.  Also $L(z_0)= L(v_{\gamma})$ and so $L(p_1)= L(v_{\gamma})$ whence $L(p_2)= L(v_{\delta})$.

Now we also have $\{z'_{n-1}, z_{m-1}\} = \{v_1,v_2\}= \{v_{\gamma},v_{\delta}\}$.  So suppose that $z_{m-1}=v_{\gamma}$ and is therefore intermediate. If we rewrite $h_{m-1}\olazmi=\smash{z_{m-1}^{-1}}k_{m-1}$ as $z_{m-1}h_{m-1}\olazmi=k_{m-1}$, then either the latter is exceptional for $F(A^*,B^*,C^*_-) \cap F(B^*,C^*)$ or $h_{m-1}=1 = k_{m-1}$ and $z_{m-1} = \olazmi$.  However if the latter holds, then $L(z_{m-1})=L(v_{\gamma}) < L(p_2) = L(v_{\delta}) = L(z_m)$ which contradicts $z_{m-1} = \olazmi$. So only the former can hold, but then since $L(z_{m-1})=L(v_{\gamma}) < L(p_2)$ we must have $z_{m-1}=p_1, \overleftarrow{z_m}=p_2$ and $h_{m-1}=p_0, k_{m-1}=q$.  However we also have $z_{m-1}=v_{\gamma} = z_0$ which contradicts $z_0=p_1(\orai)$.  The only possibility left is that $h'_{n-1}=h_{m-1}, k'_{n-1}=k_{m-1}$ and $z'_{n-1}=z_{m-1}$.  We obtain the same conclusion if $z'_{n-1}=v_{\gamma}$.

 As usual, the argument can be iterated and, if $m=n$, we get all the way to $h'_{i+1}=h_{i+1},  k'_{i+1}=k_{i+1}$ and $z'_{i+1}=z_{i+1}$ giving $h=h'$ which is clearly contradictory.  The problem remaining is when $m \neq n$ and we can assume that $m<n$. Then $$gh_0b^ip_0b^{-1}h'_{i+1} \ldots b^{-1}h_mg^{-1} = k_0b^iqb^{-1}k'_{i+1} \ldots b^{-1}k_m$$ becomes 
\begin{multline*}gh_0b^ip_0b^{-1}h'_{i+1}\ldots b^{-1}h'_{n-m+i}b^{-1}h_{i+1} \ldots b^{-1}h_{m-1}b^{-1}h_mg^{-1} 
\\= k_0b^iqb^{-1}k'_{i+1}\ldots b^{-1}k'_{n-m+i}b^{-1}k_{i+1} \ldots b^{-1}k_{m-1}b^{-1}k_m.\end{multline*} If we now conjugate both terms of the counterpair by $gh_0b^ip_0=k_0b^iqp_2$ we obtain 
\begin{align*}
p_2\hat hp_2^{-1} = p_2b^{-1}h_{i+1} \ldots b^{-1}h_mh_0b^ip_0p_2^{-1}&= b^{-1}k_{i+1} \ldots b^{-1}k_mk_0b^iq = \hat k\\
p_2\hat h'p_2^{-1} = p_2b^{-1}h'_{i+1} \ldots b^{-1}h'_{n-m+i}\hat hp_2^{-1} &= b^{-1}k'_{i+1} \ldots \ldots b^{-1}h'_{n-m+i} \hat k\tag*{\hbox{and}}\\ 
p_2b^{-1}h'_{i+1} \ldots b^{-1}h'_{n-m+i}p_2^{-1} &= b^{-1}k'_{i+1} \ldots \ldots b^{-1}k'_{n-m+i}.\tag*{\hbox{and hence}}\end{align*}
The desired contradiction now follows in the usual way from \fullref{Claim 4.2}. \end{proof}

\begin{clm}\label{Claim 4.5} There do not exist counterpairs $ghg^{-1} =k, gh'g^{-1}=k'$ satisfying $l_b(g)=0$ and {\rm min}$\{l_b(h),l_b(h')\} >0$ such that $\rho_b(h)=1= \rho_b(h')$. \end{clm}
\begin{proof} If a counterpair exists, then it fails to satisfy the hypotheses of \fullref{Claim 4.4}.  It must therefore have, without loss of generality, the form \begin{align*}\quad gh_0bh_1 \!\ldots bh_ib^{-1} \!\!\ldots b^{-1}h_{m-1}b^{-1}h_mg^{-1} &\!=\! k_0bk_1\!\ldots bk_ib^{-1}\!\!\ldots b^{-1}k_{m-1}b^{-1}k_m\\
gh'_0b^{-1}h'_1\ldots b^{-1}h'_jb\ldots bh'_{n-1}bh'_ng^{-1} &\!=\! k'_0b^{-1}k'_1\ldots b^{-1}k'_jb\ldots bk'_{n-1}bk'_n.\tag*{\hbox{\rlap{and}}}\end{align*}  The resulting Normal Form equalities from the first member of the counterpair show that the intersection $F(A^*,B^*,C^*_-) \cap F(B^*,C^*)$ is exceptional and
$gh_0bh_1 \ldots bh_i=k_0bk_1 \ldots bk_iw_i$, with $\{\overleftarrow{z_{i-1}},w_i\} = \{p_1,p_2\}$. If we substitute for $g$, then we obtain a counterpair
\begin{align*}
w_i\hat hw_i^{-1} &= w_ib^{-1}h_{i+1} \ldots h_{m-1}b^{-1}h_mh_0bh_1 \ldots bh_iw_i^{-1}\\
&= b^{-1}k_{i+1} \ldots k_{m-1}b^{-1}k_mk_0bk_1 \ldots bk_i =\hat k, \\
w_i\hat h'w_i^{-1}&= w_ih_i^{-1}b^{-1} \ldots b^{-1}h_0^{-1}h'_0b^{-1}h'_1b^{-1}\ldots b^{-1}h'_jb\ldots bh'_nh_0bh_1 \ldots bh_iw_i^{-1} \\
&= k_i^{-1}b^{-1} \ldots b^{-1}k_0^{-1}k'_0b^{-1}k'_1b^{-1}\ldots b^{-1}k'_jb\ldots bk'_nk_0bk_1 \ldots bk_i = \hat k'\end{align*} since we have just conjugated the original counterpair.  

 We cannot exclude the possibility that $\smash{w_i\hat hw_i^{-1}}$ is not in reduced form and it is possible that $\smash{\rho_b(\hat h)} = 0$, and even that $l_b(h) =0$.  However $\smash{w_i\hat h'w_i^{-1}}$ is in reduced form and so we have $\rho_b(\smash{\hat h}) \leq 1, \rho_b(\smash{\hat h'}) = 1$ and we reduce to one of \fullref{Claim 4.1}, \fullref{Claim 4.2} or Claim 4.4 as appropriate. \end{proof}

Finally we are ready to verify the our overall conclusion that there are no counterpairs, having verified this assertion for three initial cases.
\begin{clm}\label{Claim 4.6} There do not exist counterpairs $ghg^{-1} =k, gh'g^{-1}=k'$ satisfying $l_b(g)=0$ and {\rm min}$\{l_b(h),l_b(h')\} >0$. \end{clm}

\begin{proof} In Claims \ref{Claim 4.2}, \ref{Claim 4.3} and \ref{Claim 4.4},  we have verified that there are no counterpairs satisfying $l_b(g)=0$ and {\rm min}$\{l_b(h),l_b(h')\} >0$ under any of the additional hypotheses $\rho_b(h)+ \rho_b(h') =0$, $\rho_b(h)+ \rho_b(h') =1$, and $\rho_b(h)+ \rho_b(h') =2$ with $\rho_b(h)= \rho_b(h') =1$.  This leaves us with the following cases.

{\bf Case 4.6.1}\qua $\rho_b(h)+ \rho_b(h') \ge 2$ and both are even.
 
{\bf Case 4.6.2}\qua $\rho_b(h)+ \rho_b(h') \ge 3$ and one is odd and the other is even.

{\bf Case 4.6.3}\qua $\rho_b(h)+ \rho_b(h') \ge 4$ and both are odd.

We assume that we have a counterpair with $\rho_b(h)+ \rho_b(h') \ge 2$ and minimal where, without loss of generality, we can assume that $\rho_b(h) \leq \rho_b(h')$. We need to split this into two subcases.

{\bf Case 4.6.1a}\qua Let $\rho_b(h)=0$ so that $\rho_b(h') \ge  2$.  Then we can write 
\begin{align*}ghg^{-1} &= gh_0bh_1 \ldots bh_m= k_0bk_1 \ldots bk_m=k\\
gh'g^{-1} &=gh'_0bh'_1 \ldots bh'_jb^{-1}h'_{j+1} \ldots b^{-1}h'_lbh'_{l+1} \ldots bh'_ng^{-1}\tag*{\hbox{and}} \\
&= k'_0bk'_1 \ldots bk'_jb^{-1}k'_{j+1} \ldots b^{-1}k'_lbk'_{l+1} \ldots bk'_n=k'\end{align*}
where the sign changes displayed in $h'$and $k'$ are the initial two. 

Suppose $j \leq m$; when we analyse the two systems of Normal Form inequalities, one of two possibilities occurs.  The first is that we can iteratively obtain equalities $\smash{h'_0} = h_0$, $\smash{k'_0 = k_0}$, $\smash{z'_0 = z_0}$, $\smash{h'_1=h_1,k'_1 = k_1,z'_1=z_1}, \ldots, \smash{h'_{j-1} = h_{j-1}}$, $\smash{k'_{j-1} = k_{j-1}}$, $\smash{z'_{j-1}}=z_{j-1}$ because the successive equalities derived by elimination hold freely.  In this case we conjugate by $gh_0bh_1 \ldots bh_{j-1}b=k_0bk_1 \ldots bk_{j-1}b{\overleftarrow {z_{j-1}}}$ to obtain a new counterpair of the form $\smash{{\overleftarrow {z_{j-1}}} \hat h{\smash{\overleftarrow {z_{j-1}}}}^{\!-1} = \hat k}$,  $\smash{{\overleftarrow {z_{j-1}}} \hat h'{\smash{\overleftarrow {z_{j-1}}}^{\!-1}} = \hat k'}$ where $\hat h = h_jbh_{j+1}\! \ldots bh_m h_0bh_1\!\ldots bh_j$, $\hat h' = h'_jb^{-1}h'_{j+1} \!\ldots b^{-1}h'_lb \ldots bh'_nh_0bh_1 \!\ldots bh_j$, and similarly for $\smash{\hat k}$ and $\smash{\hat k'}$. In both cases no new sign changes are introduced and $h'$ has been stripped of its initial sign change.  Thus $\rho_b(\hat h)=\rho_b(h) =0$ and $\rho_b(\hat h') = \rho_b(h') - 1$ and we contradict minimality  using \fullref{Claim 4.2} if  $\rho_b(h') =2$.

The second possibility is that our sequence of free equalities breaks down and we obtain an exceptional equality for $F(A^*,B^*) \cap F(A^*_+,B^*,C^*)$ of the form $$z_f^{-1}k_f^{-1}k'_fz'_f = h_f^{-1}h'_f,$$ for some $f \leq j-1$. In particular $\{z_f, z_f'\} = \{v_1, v_2\}$.  Here we use $gh_0bh_1 \ldots bh_f=k_0bk_1 \ldots bk_fz_f$ to obtain a conjugate counterpair $$z_f\hat h z_f^{-1} = \hat k, z_f\hat h' z_f^{-1} = \hat k'.$$  Again we simply permute $h$ to obtain $\hat h$ whereas 
\begin{align*}z_f\hat h'z_f^{-1} &= z_fh_f^{-1}h'_fbh'_{f+1} \ldots bh'_jb^{-1} \ldots b^{-1}h'_lb \ldots bh'_nh_fz_f^{-1}\\
&= k_f^{-1}k'_fbk'_{f+1} \ldots bk'_jb^{-1} \ldots b^{-1}k'_lb \ldots bk'_nk_f =\hat k'.\end{align*}   From this last equality we deduce the two equalities \begin{align*}z_fh_f^{-1}h'_fbh'_{f+1} \ldots bh'_jb^{-1} \ldots b^{-1}h'_l &= k_f^{-1}k'_fbk'_{f+1}{z'}_l^{-1} \ldots bk'_jb^{-1} \ldots b^{-1}k'_l \\ z'_lbh'_{l+1} \ldots bh'_nh_0bh_1 \ldots bh_jz_f^{-1} &= bk'_{l+1} \ldots bk'_nk_0bk_1 \ldots bk_j.\tag*{\hbox{and}}\end{align*}  However $\{z_f, z'_f\} = \{v_1, v_2\} = \{z_l, z_l'\}$ and $u=v_1^{-1}v_0v_2$ and this means that $z_l' = z_f$, $\smash{z_l' = v_0z_fv_2^{-1}}$ or $\smash{z_l' = v_0^{-1}z_fu}$.  Recalling that $\smash{v_1= v_0v_2u^{-1}}$ and hence $\smash{v_2= v_0^{-1}v_1u}$, we can transform $\smash{z_f\hat h' z_f^{-1}= \hat k'}$ into an equality $\smash{z_f\tilde h' z_f^{-1}z_f\breve h'z_f^{-1}= \breve k'}$ where $\rho_b(\tilde h') = \rho_b(h')- 1$ and $\rho_b(\breve h') =0$. By minimality, using \fullref{Claim 4.1} if $\rho_b(h')=2$, it follows that $h$ commutes with both $\tilde h'$ and  $\breve h'$ and hence with $\hat h$, which is the contradiction we require to conclude the argument when $j \leq m$.

If $j>m$ we have essentially the same possibilities as before, save that when we have free equalities we might need to rotate $h$ several times before we reach either the situation when we get free equalities involving $h'_{j-1}, k_{j-1}$ and $z'_{j-1}$ or we obtain an exceptional equality before $h'_{j-1}, k_{j-1}$ and $z'_{j-1}$ are involved. (An alternative view is to say that we make a minimal choice of $n=l_b(h')$ and then, if we obtain $m$ free equalities, we replace our original pairs $(h,k)$ and $(h',k')$ by  $(h,k)$ and $(h^{-1}h',k^{-1}k')$.)
 
{\bf Case 4.6.1b}\qua Let $\rho_b(h)\ge 2$ so that $\rho_b(h)+\rho_b(h') \ge  4$.  If we assume, without loss of generality that $i \ge j$, then the argument given above can be repeated more or less verbatim. If $i>j$, any conjugation used will preserve $\rho_b(h)$ while if $i=j$ and the free equalities are valid as far as $i-1=j-1$, the conjugation used will reduce $\rho_b(h)+ \rho_b(h')$ by $2$. 

{\bf Case 4.6.2}\qua Let $\rho_b(h)+ \rho_b(h') \ge 3$ where one is odd and the other is even.

Without loss of generality, we may suppose that $\rho_b(h)$ is even and $\rho_b(h')$ is odd, not excluding the possibility that $\rho_b(h)=0$, in which case $\rho_b(h') \ge 3$.  Also, by inverting $ghg^{-1}=k$, if necessary, we can assume that the two terms of our counterpair have the same initial exponent for $b$ 

We proceed much as in Case 4.6.1. However, there we attacked the terms of our counterpair by obtaining a sequence of equalities $z_0^{-1}k_0^{-1}k'_0z'_0 = , z_1^{-1}k_1^{-1}k'_1z'_1=h_1^{-1}h'_1 ,\ldots$ until we found one that did not hold freely.  This time we have three sequences of such equalities because of the fact that $\rho_b(h')$ is odd -- the initial three equalities are  $z_0^{-1}k_0^{-1}k'_0z'_0 = h_0^{-1}h'_0, z_0^{-1}k_0^{-1}k_m^{-1}z_m = h_0^{-1}h_m^{-1}, {z'_n}^{-1}k'_nk'_0z'_0 = h'_nh'_0$.  If we can run these free inequalities until we reach a sign change in $h$ or $h'$ (if $\rho_b(h)$ =0, then only $h'$ is a possibility as discussed in the previous case), then conjugation will replace our original counterpair by a counterpair with fewer total sign changes. The conjugation will cycle positive occurrences of $b$ from the front of $h$ to the back of $h$ and will actually cancel occurrences of $b$ that occur in $h'$.  The other alternative is that we reach a point at which some equality is exceptional for $F(A^*,B^*) \cap F(A^*_+,B^*,C^*)$, in which case there will be a conjugate counterpair of the form $\smash{z_f\hat hz_f^{-1} = \hat k}$ or $\smash{z_f\hat h'z_f^{-1} = \hat k'}$ such that one or other (or possibly both) will decompose into two counterpairs, each of which contains fewer sign changes than our original.  The resulting commutativity derived from our assumption of minimality then yield the required contradiction.

{\bf Case 4.6.3}\qua Let $\rho_b(h)+ \rho_b(h') \ge 4$ where both are odd.   Initially let us assume that $h$ and $h'$ have positive exponent on the respective initial occurrences of $b$.  Then we are in a situation similar to that considered in \fullref{Claim 4.4} where we attack both terms of our counterpair  from the front and back. As we noted proving claim 4.4 there are potentially six apparently distinct sequences of equalities obtained by elimination from the Normal Form equalities.  Broadly our argument is the same as that for Case 4.6.2.  Either we can generate free inequalities right up to the point at which we reach a sign change, in which case conjugation will provide us with a new counterpair with fewer total sign changes or at some point, we produce an equality that is exceptional for $F(A^*,B^*) \cap F(A^*_+,B^*,C^*)$.  But again there will be a conjugate counterpair, one of whose terms will contain enough sign changes to allow us to decompose it into two factors, each containing fewer sign changes than the original term and we have the same commutativity conclusion.

It remains only to note that the argument of \fullref{Claim 4.5} in fact carries over verbatim to the present situation and allows us to drop out provisional hypothesis concerning the exponents of the respective initial occurrences of $b$. \end{proof}
 
We end this section by observing that the sequence of Claims \ref{Claim 4.1}--\ref{Claim 4.6} completes the proof of our main result, save that we have to verify the Propositions stated in the next section and which were used above.

\section{Technical results}\label{sec5}

As noted in \fullref{sec4} just prior to the application of the results we are about to prove, the material in this section parallels Proposition 6.2 of \cite{C} and we shall employ the methods, terminology and notation described there. Also, as noted at the start of \mbox{\fullref{sec4}}, we can assume that all of $A^*_+, A^*_-, C^*_+, C^*_-$ are nonempty.   As in \S 6 of \cite{C}, our initial standpoint is that we are given the exceptional intersection $$F(A^*,B^*) \cap F(A^*_+,B^*,C^*) = \langle u \rangle * F(A^*_+,B^*) = \langle v \rangle * F(A^*_+,B^*)$$ with $u = v_1^{-1}v_0v_2$, where $v_1,v_2$ are not both trivial, and in turn  $F(A^*,B^*,C^*_-) \cap F(A^*_+,B^*,C^*)$ is also exceptional with basic exceptional equality $s=t$ where $s$ is the $a_{\kappa}$--core of $u$ and $t \equiv u_1^{-1}vu_2^{-1}$, where $u \equiv u_1su_2$. We also write $t \equiv t_1\bar t t_2$ where $\bar t$ is
the $c_{\nu}$--core of $t$.  

We shall deal with two specific additional case assumptions, in each instance proving a result similar to Proposition 6.2 of \cite{C} (which is itself proved under its own set of assumptions additional to the basic standpoint of \S 6 of \cite{C}).

\begin{caa}\label{CAA} In $v_1^{-1}v_0v_2$, neither $v_1$ nor $v_2$ is intermediate.  Since both $v_1$ and $v_2$ lie in $F(A^*_+,B^*,C^*_+)$, this amounts to saying that both involve at least one of the two extremal generators $a_{\lambda}, c_{\nu}$.\end{caa}

Our first step is to prove an analogue of Lemma 6.1 of \cite{C}.

\begin{lem}\label{lemma5.1} Let $$F(A^*,B^*) \cap F(A^*_+,B^*,C^*) = \langle u \rangle * F(A^*_+,B^*) = \langle v \rangle * F(A^*_+,B^*)$$ with $u = v \equiv v_1^{-1}v_0v_2$ where \fullref{CAA} holds. Then:
\begin{enumerate}
\item [\rm(a)] A cyclically reduced word of the form $h^{-1}w^{-1}kz$, where $w\in L$ and $z\in U$ are both nontrivial of type $(A^*:C^*)$ and $h \in F(A^*,B^*)$, $k \in F(B^*,C^*)$ with $h,k$ nontrivial, cannot (cyclically) contain two disjoint Gurevich subwords.
\item [\rm(b)] A cyclically reduced word of the form $k^{-1}w^{-1}hz$, where $w\in L$ and $z\in U$ are both nontrivial of type $(C^*:A^*)$ and $h \in F(A^*,B^*)$, $k \in F(B^*,C^*)$ with $h,k$ nontrivial, cannot (cyclically) contain two disjoint Gurevich subwords.
\item [\rm(c)] A cyclically reduced word of the form $h^{-1}w^{-1}h'z$, where $w\in L$ and $z\in U$ are both nontrivial of type $(C^*:C^*)$ and also $h,h' \in F(A^*,B^*)$ are nontrivial, cannot (cyclically) contain two  disjoint Gurevich subwords.
\item [\rm(d)] A cyclically reduced word of the form $k^{-1}w^{-1}k'z$, where $w\in L$ and $z\in U$ are both nontrivial of type $(A^*:A^*)$ and also $k,k' \in F(B^*,C^*)$ are nontrivial cannot (cyclically) contain two  disjoint Gurevich subwords.
\end{enumerate}
\end{lem}
\begin{proof} It suffices to prove (a) and (c) since (b) is just a dual rewording of (a) and (d) is a dual rewording of (c). 
 
(a)\qua Suppose we have two disjoint Gurevich subwords of $h^{-1}w^{-1}kz$; then there are two disjoint extremal Gurevich subwords. Now neither extremal Gurevich subword can be a subword of any of  $h^{-1}w^{-1}, w^{-1}k,kz, zh^{-1}$, for each of these omits an essential generator. Moreover, neither extremal Gurevich subword can contain any of  $h^{-1}w^{-1}, w^{-1}k,kz, zh^{-1}$, for then its companion extremal Gurevich subword would be a subword of one of $h^{-1}w^{-1}, w^{-1}k,kz, zh^{-1}$.  It follows, therefore that an extremal Gurevich subword must take one of the four forms  $$h_1^{-1}w^{-1}k_1, w_1^{-1}kz_1, k_2zh_2^{-1}, z_2h^{-1}w_2^{-1},$$ where $w_1, w_2$ denote proper, nontrivial, initial and terminal segments of $w$ and similarly for $h,\ k$ and $z$, and that a pair must be either $\{h_1^{-1}w^{-1}k_1, k_2zh_2^{-1}\}$ or $\{w_1^{-1}kz_1, z_2h^{-1}w_2^{-1}\}$.
 
Suppose that a word of form $\smash{h_1^{-1}w^{-1}k_1}$ is an extremal Gurevich subword. Then $\smash{a_{\kappa}^{\pm 1}}$, which must be obtained from $\smash{u^{\pm 1}}$ can appear either in $\smash{h_1^{-1}}$ or in $w^{-1}$ and similarly $\smash{c_{\mu}^{\pm 1}}$ from $\smash{v_0^{\pm 1}}$ can appear either in $w^{-1}$ or $k_1$. Wherever they appear, the occurrences of $\smash{a_{\kappa}^{\pm 1}}$ and $\smash{c_{\mu}^{\pm 1}}$ will properly enclose between them, a string of syllables of $\smash{h_1^{-1}w^{-1}k_1}$ that are distinct from those containing $a_{\kappa}^{\pm 1}$ and $c_{\mu}^{\pm 1}$ and which constitute an occurrence of either $\smash{v_1^{\pm 1}}$ or $\smash{v_2^{\pm 1}}$.  This means that one of these must occur within $w^{-1}$ which contradicts the fact that neither $v_1^{\pm 1}$ nor $v_2^{\pm 1}$is intermediate. This rules out the first possibility for a pair.

The second possibility for a pair includes a word of form $\smash{w_1^{-1}kz_1}$ as an extremal Gurevich subword.  For this word, $\smash{a_{\kappa}^{\pm 1}}$ can appear only in $\smash{w_1^{-1}}$ and $\smash{c_{\mu}^{\pm 1}}$ in  $w_1^{-1}$ or $k$.  An analysis similar to the previous possibility  forces either $\smash{v_1^{\pm 1}}$ or $\smash{v_2^{\pm 1}}$ to lie within $w_1^{-1}$ which is impossible.

(c)\qua In a manner parallel to the argument for (a), we see that no member of a pair of disjoint extremal Gurevich subwords can be contained in or contain any of $h^{-1}w^{-1}, w^{-1}h',h'z, zh^{-1}$. Furthermore we cannot have a pair of extremal Gurevich subwords of the form $h_1^{-1}w^{-1}h'_1, h'_2zh_2^{-1}$ and so the only possible form for a pair is $\smash{w_1^{-1}h'z_1, z_1h^{-1}w_2^{-1}}$. Observing that $\smash{c_{\mu}^{\pm 1}}$ can appear only in $w^{-1}$ while $\smash{a_{\kappa}^{\pm 1}}$ can appear in $h^{-1},h'$ or $w^{-1}$ we see that we are forced to try to position $v_1^{\pm 1}$ or $v_2^{\pm 1}$ within $w^{-1}$, which is impossible.
\end{proof}

The following is the first of our two results that parallels Proposition 6.2. of \cite{C}

\begin{prop}\label{Proposition 5.2} Let $$F(A^*,B^*) \cap F(A^*_+,B^*,C^*) = \langle u \rangle * F(A^*_+,B^*) = \langle v \rangle * F(A^*_+,B^*),$$ where $u =v$ in $G^*$, $v$ is $v_1^{-1}v_0v_2$ and \fullref{CAA} holds. Furthermore, let the equality $wh=kz$, where $w\in L$ and $z\in U$ are both nontrivial of type $(A^*:C^*)$ and $h \in F(A^*,B^*)$, $k \in F(B^*,C^*)$, define an element of $F(A^*,B^*,C^*_-) \cap F(A^*_+,B^*,C^*)$.  Then the element defined by $wh=kz$ is non-exceptional and the equality holds freely in $F(A^*_+,B^*,C^*_-)$ -- in particular, $h=k=1$ and $w \equiv z$ is intermediate.
\end{prop}
\begin{proof} Suppose, by way of contradiction, that the element defined is exceptional so that use of the basic exceptional relation is required.   

{\bf Case 5.2.1}\qua Suppose that $h$ and $k$ are nontrivial so that $h^{-1}w^{-1}kz$ is cyclically reduced as written. If we apply \fullref{lemma5.1}(a), then it remains only to show that the cyclically reduced form $h^{-1}w^{-1}kz$ cannot be a cyclic rearrangement of $(u^{-1}v_1^{-1}v_0v_2)^{\pm 1}$.  As before we look to see where the extremal generators are situated.  In particular we observe that $u^{\pm 1}$ either coincides with $h^{-1}$ or is a syllable of $w^{-1}$.  Similarly $v_0^{\pm 1}$ either coincides with $k$ or is a syllable of $w^{-1}$. No matter which possibility occurs, we finish up, as previously, trying to position $v_1^{\pm 1}$ or $v_2^{\pm 1}$ within $w^{-1}$.

{\bf Case 5.2.2}\qua Suppose that $h=1$ and $k$ is nontrivial. Now $w$ and $z$ may have a common terminal segment which will be cancelled in obtaining the cyclically reduced form of $kzw^{-1}$; notice however that no occurrences of extremal generators -- and $a_{\kappa}$ must appear in $w$ and $a_{\lambda}$ in $z$ -- will be cancelled.  Then we can write $w\equiv w_1w_2$ and $z\equiv z_1z_2$ where $w_2\equiv z_2$ is the maximal common terminal segment of $w$ and $z$, with $w_1,z_1$ nontrivial.  Then the resulting cyclically reduced word will be either of the form $kz'h'^{-1}w'^{-1}$ or $kz'k'^{-1}w'^{-1}$, depending on the exact nature of $w_2$ and $z_2$ in relation to $w$ and $z$, with $h'$, respectively $k'$, nontrivial. Then we can apply either \fullref{lemma5.1} (a) or (d) to deduce that the only possibility for this word is that it is a cycle of $(u^{-1}v_1^{-1}v_0v_2)^{\pm 1}$. The argument for Case 5.2.1 disposes of the possibility that we have $kz'h'^{-1}w'^{-1}$. 

 To finish this case we verify that $kz'k'^{-1}w'^{-1}$ cannot be a cycle of $\smash{(u^{-1}v_1^{-1}v_0v_2)^{\pm 1}}$.  This time $u^{\pm 1}$ must be a syllable of $w'^{-1}$ while $v_0^{\pm 1}$ can be $k,k'^{-1}$ or a syllable of $w'^{-1}$; but of course we then have to position  $v_1^{\pm 1}$ or $v_2^{\pm 1}$ within $w'^{-1}$. 

{\bf Case 5.2.3}\qua Suppose that $h$ is nontrivial and $k=1$ . This is clearly dual to Case 5.2.2, by considering $w^{-1}zh^{-1}$ and the consequent cyclically reduced form.

{\bf Case 5.2.4}\qua Suppose that $h=k=1$.  This time we simply examine $w^{-1}z$ but have to allow for both common initial segments and common terminal segments, observing that both will have to be intermediate words. The resulting cyclically reduced form will fall into one of the previous categories we have considered. 
\end{proof}

For the next three results we replace \fullref{CAA} by the following. 

\begin{caa}\label{CAB} In $v_1^{-1}v_0v_2$, both $v_1$ and $v_2$ are nontrivial and intermediate. \end{caa}
 
\begin{lem} Let $$F(A^*,B^*) \cap F(A^*_+,B^*,C^*) = \langle u \rangle * F(A^*_+,B^*) = \langle v \rangle * F(A^*_+,B^*)$$ with $u = v_1^{-1}v_0v_2$ where \fullref{CAB} holds. Then $F(A^*,B^*,C^*_-) \cap F(B^*,C^*)$ is exceptional.  Moreover if the basic exceptional equality is $p_1p_0p_2^{-1} = q$ , then, under the conventions described prior to Proposition 5.5 of \cite{C}, $p_1 = v_1, p_0 = u, p_2 = v_2$ and $q=v_0$. \end{lem}
\begin{proof} This is immediate from the definitions involved. \end{proof}

We use the \textit{syllable length function} $L$ applicable to words of $F(A^*_+,B^*,C^*)$ or 
$F(A^*,B^*,C^*_-)$, defined as the number of syllables of $z$. The terms ``syllable'' and ``syllable length'' are defined at the end of \S 5 of \cite{C} but unfortunately the notation $L$ for this was not specifically defined there -- the reader should refer to \fullref{add2}. 

\begin{lem}\label{Lemma 5.4} Let $$F(A^*,B^*) \cap F(A^*_+,B^*,C^*) = \langle u \rangle * F(A^*_+,B^*) = \langle v \rangle * F(A^*_+,B^*)$$ with $u = v_1^{-1}v_0v_2$ where \fullref{CAB} holds. Then :
\begin{enumerate}
\item[\rm(a)] Let $h^{-1}w^{-1}kz$ be a cyclically reduced word, where $w\in L$ and $z\in U$ are both nontrivial of type $(A^*:C^*)$ and $h \in F(A^*,B^*)$, $k \in F(B^*,C^*)$ with $h,k$ nontrivial. Suppose that $\hbox{min}\{L(w),L(z)\} \leq \hbox{min}\{L(v_1),L(v_2)\}$.  If $h^{-1}w^{-1}kz$ contains a pair of disjoint extremal Gurevich subwords then such a pair must be of the form $h_1^{-1}w^{-1}k_1$ and $k_2zh_2^{-1}$ where $h_1,k_1$ are proper initial segments of $h$ and $k$, $h_2,k_2$ are proper terminal segments of $h$ and $k$ and the following hold:
\begin{enumerate} 
\item[\rm(i)] $h_1^{-1}w^{-1}k_1$ is of the form either
{\rm (1)} $u_1(a_{\kappa},a_{\lambda} )^{-1}v_1^{-1}v_{01}(c_{\mu},c_{\nu})$ with $w= v_1$ or  
{\rm (2)} $u_2(a_{\kappa},a_{\lambda})v_2^{-1}v_{02}^{-1}(c_{\mu},c_{\nu})$  with $w= v_2$;

\item[\rm(ii)] $k_2zh_2^{-1}$ is of the form either 
{\rm (3)} $v_{02}(c_{\mu},c_{\nu})v_2u_1(a_{\kappa},a_{\lambda} )^{-1}$  with $v_2 = z$ or 
{\rm (4)} $v_{01}(c_{\mu},c_{\nu})^{-1}v_1u_1(a_{\kappa},a_{\lambda})$ with $v_1 = z$.
\end{enumerate} 
In the above $u_1$, $v_{01}$, $u_2$, $v_{02}$ are appropriate initial or terminal segments of $u$ and $v_0$. 

\item[\rm(b)] Let $h^{-1}w^{-1}h'z$ be a cyclically reduced word, where $w\in L$ and $z\in U$ are both nontrivial of type $(C^*:C^*)$ and $h \in F(A^*,B^*)$, $k \in F(B^*,C^*)$ with $h,h'$ nontrivial. Suppose that $\hbox{min}\{L(w),L(z)\} \leq  \hbox{min}\{L(v_1),L(v_2)\}$. Then $h^{-1}w^{-1}h'z$ cannot (cyclically) contain two disjoint Gurevich subwords.

\item[\rm(c)] Let $k'^{-1}w^{-1}kz$ be a cyclically reduced word, where $w\in L$ and $z\in U$ are both nontrivial of type $(A^*:A^*)$ and $k,k' \in F(B^*,C^*)$ with $h,h'$ nontrivial. Suppose that $\hbox{min}\{L(w),L(z)\} \leq \hbox{min}\{L(v_1),L(v_2)\}$. Then $k'^{-1}w^{-1}z$ cannot (cyclically) contain two disjoint Gurevich subwords.
\end{enumerate}
\end{lem}
\begin{proof} We omit the proof of (c) since the statement is the dual of (b).

{\rm (a)}\qua A pair of extremal Gurevich subwords must be either $\{h_1^{-1}w^{-1}k_1, k_2zh_2^{-1}\}$ or $\{w_1^{-1}kz_1, z_2h^{-1}w_2^{-1}\}$, as in \fullref{lemma5.1}. Since $a_{\kappa}$ and $a_{\lambda}$ occur together in $u$, when we inspect our candidate pair $\{h_1^{-1}w^{-1}k_1,k_2zh_2^{-1}\}$ we see that $u$ cannot be matched against a syllable of $w$ or $z$ and hence we must have both $a_{\kappa}$ and $a_{\lambda}$ together in $h_1^{-1}$ and $h_2^{-1}$ respectively.  Similar remarks apply to $c_{\mu}$ and $c_{\nu}$ and it follows that for pairs $\{h_1^{-1}w^{-1}k_1,k_2zh_2^{-1}\}$, the possibilities are those listed above.

An analysis of the possibilities for pairs $\{w_1^{-1}kz_1, z_2h^{-1}w_2^{-1}\}$ yields the following:
\begin{align*}w_1^{-1}kz_1&\equiv u_1(a_{\kappa})^{-1}v_1^{-1}v_0v_2u_2(a_{\lambda})^{-1}&&\text{ with }v_0=k;\\ 
w_1^{-1}kz_1 &\equiv u_2(a_{\kappa})v_2^{-1}v_0^{-1}v_1u_1(a_{\lambda})&&\text{ with }v_0^{-1}=k;\\
z_2h^{-1}w_2^{-1}&\equiv v_{02}(c_{\nu})v_2u^{-1}v_1^{-1}v_{01}(c_{\mu})&&\text{ with }u^{-1} = h^{-1};\\
z_2h^{-1}w_2^{-1}&\equiv v_{01}(c_{\nu})^{-1}v_1uv_2^{-1}v_{02}(c_{\mu})^{-1}&&\text{ with }u = h^{-1},
\end{align*}
where we extend our convention about denoting initial and terminal subscripts of words in this case to $u$ and $v_0$ by writing $u_1, u_2$ and $v_{01}, v_{02}$ respectively. However in each case we observe $v_1^{\pm 1}$ and $v_2^{\pm 1}$ as proper subwords of either $w$ or $z$ contradicting $\hbox{min}\{L(w),L(z)\} \leq \hbox{min}\{L(v_1),L(v_2)\}$.
 
{\rm (b)}\qua In a manner parallel to the argument for \fullref{lemma5.1} (c),it follows the only possible form for a pair is $w_1^{-1}h'z_1, z_1h^{-1}w_2^{-1}$. 
The options are: 
\begin{align*}w_1^{-1}h'z_1 &\equiv v_{02}(c_{\mu})v_2u^{-1}v_1^{-1}v_{01}(c_{\nu})&&\text{ with }u^{-1} = h';\\
w_1^{-1}h'z_1 &\equiv v_{01}(c_{\mu})^{-1}v_1uv_2^{-1}v_{02}(c_{\nu})^{-1} ) &&\text{ with }u = h';\\
 z_2h^{-1}w_2^{-1}&\equiv v_{02}(c_{\nu})v_2u^{-1}v_1^{-1}v_{01}(c_{\mu})^{-1} )&&\text{ with }u^{-1} = h^{-1};\\
 z_2h^{-1}w_2^{-1}&\equiv v_{01}(c_{\nu})^{-1}v_1uv_2^{-1}v_{02}(c_{\mu})^{-1} )&&\text{ with }u = h.
\end{align*}
 However, in each case the length inequality is contradicted. \end{proof}

\begin{prop}\label{Proposition 5.5} Let $$F(A^*,B^*) \cap F(A^*_+,B^*,C^*) = \langle u \rangle * F(A^*_+,B^*) = \langle v \rangle * F(A^*_+,B^*)$$ with $u = v_1^{-1}v_0v_2$ where \fullref{CAB} holds. Furthermore let the equality $wh=kz$, where $w\in L$ and $z\in U$ are both nontrivial of type $(A^*:C^*)$ and $h \in F(A^*,B^*)$, $k \in F(B^*,C^*)$ , define an element of $F(A^*,B^*,C^*_-) \cap F(A^*_+,B^*,C^*)$. 
 
If $\hbox{min}\{L(w),L(z)\} \leq \hbox{min}\{L(v_1),L(v_2)\}$, then one of the following holds:
\begin{enumerate}
\item[\rm(a)] The element defined by $wh=kz$ is non-exceptional and the equality holds freely -- in particular, $h=k=1$ and $w \equiv z$ is intermediate.

\item[\rm(b)] $v_1, v_2$ are distinct and nontrivial and $h^{-1}w^{-1}kz$ is a cycle of $(u^{-1}v_1^{-1}v_0v_2)^{\pm 1}$.  In particular $w \equiv v_1, z \equiv v_2, h\equiv u, k\equiv v_0$, in other words $wh=kz$ is precisely $v_1u = v_0v_2$, or, similarly, $wh=kz$ is precisely $v_2u^{-1}=v_0^{-1}v_1$;

\item[\rm(c)] $v_1 = v_2 \equiv \tilde v$ is nontrivial and $h^{-1}w^{-1}kz$ is a cycle of $(u^{-l}\tilde v^{-1}v_0^l\tilde v)^{\pm 1}$ for some nonzero integer $l$. In particular, $w \equiv \tilde v \equiv z$ and $h=u^l, k= v_0^l$.
\end{enumerate}
\end{prop}
\begin{proof} If some extremal generator does not appear in $wh=kz$, then the equality must hold freely in the Magnus subgroup omitting this generator and (a) follows. So we can assume that all four do appear. 
 
{\rm (i)}\qua Suppose, firstly, that $h,k \neq 1$ so that $h^{-1}w^{-1}kz$ is cyclically reduced and (a) cannot hold.  Then either $h^{-1}w^{-1}kz$ is a cycle of $(u^{-1}v_1^{-1}v_0v_2)^{\pm 1}$ or $h^{-1}w^{-1}kz$ contains a pair of disjoint extremal Gurevich subwords. 

Let the former occur. Since $u$ and $v_0$ contain, respectively, $a_{\lambda}$ as well as $a_{\kappa}$ and $c_{\nu}$ as well as $c_{\mu}$, $w$ and $z$ have to be subwords of $v_1,v_2$ or their inverses, one to each. Since $L(w) + L(z) = L(v_1)+L(v_2)$, we then obtain either (b), or (c) with $l=1$.  Suppose, on the other hand, that $h^{-1}w^{-1}kz$ contains a pair of disjoint extremal Gurevich subwords. By \fullref{Lemma 5.4}(a), one of $w \equiv v_1,w\equiv v_2,z\equiv v_1, z\equiv v_2$ must hold. Suppose, for instance, $w\equiv v_1$; then $h=w^{-1}kz$ must define an exceptional element of $F(A^*,B^*)\cap F(A^*_+,B^*,C^*)$. By Proposition 5.1 of \cite{C} applied to $F(A^*,B^*) \cap F(A^*_+,B^*,C^*)$ either (b) or (c) holds.  Similar arguments apply in the remaining cases, using, in addition, the fact that $F(A^*,B^*,C^*_-) \cap F(B^*,C^*)$ is exceptional.

{\rm (ii)}\qua Suppose that $h=1$ and $k \neq 1$; as noted, the equality cannot hold freely and we shall show that it cannot in fact occur. We find ourselves in a position similar to that of \fullref{Proposition 5.2} where the cyclically reduced form of $kzw^{-1}$ is obtained by cancelling a common terminal segment of $w$ and $z$. As previously, this common initial segment must be intermediate and so the occurrences of $a_{\kappa}$ and $a_{\lambda}$, which necessarily appear in $w$ and $z$, respectively will not be cancelled.  Then, depending on the exact nature of common terminal segment cancelled, the resulting cyclically reduced word will be either of the form $kz'h'^{-1}w'^{-1}$ with $w'$, $z'$ also both of type $(A^*:C^*)$, or $kz'k'^{-1}w'^{-1}$, with $w'$, $z'$  both of type $(A^*:A^*)$, and $h'$, respectively $k'$, nontrivial.  
 
Suppose that we get $kz'h'^{-1}w'^{-1}$; since $k,h' \neq 1$ this is cyclically reduced.  By repeating the argument for Case (i), we deduce that $w'\equiv v_1$ and $z' \equiv v_2$ or vice-versa. However we also know that $L(w')<L(w), L(z')<L(z)$, since the final syllables of $w$ and $z$ must have been completely cancelled and so the length inequality is contradicted and this situation cannot occur.
 
If we have $kz'k'^{-1}w'^{-1}$, then this too is cyclically reduced. It cannot be a cycle of $(u^{-1}v_1^{-1}v_0v_2)^{\pm 1}$ since occurrences of $a_{\kappa}$ and $a_{\lambda}$ are separated by $k$ and $k'$.  We again have $L(w')<L(w), L(z')<L(z)$ since we ``raided''  the final syllables of $w$ and $z$ to obtain $k'$ and thus $\hbox{min}\{L(w'),L(z')\} \leq \hbox{min}\{L(v_1),L(v_2)\}$.  By \fullref{Lemma 5.4}(c), $kz'k'^{-1}w'^{-1}$ cannot contain two disjoint Gurevich subwords. This completes the elimination of all possibilities. 
 
The remaining cases (iii), when $h\neq 1$ and $k = 1$ and (iv) $h=k=1$ are disposed of similarly. \end{proof}

\section{Final remarks}\label{sec6}

One can derive a slightly more general conclusion from \fullref{thm2}.  We begin with a simple Lemma.

\begin{lem}\label{lemma6.1} Let $G = \langle X : r=1 \rangle$, where $r$ is cyclically reduced, be a one-relator group. Further let $M = F(S),N = F(T)$ be Magnus subgroups of $G$ and $g,g'$ be elements of $G$ and suppose that $gMg'^{-1}\cap N$ is nonempty. Then:
\begin{enumerate} \item [\rm(a)] For any element $k \in  gMg'^{-1}\cap N$, $$(gMg^{-1} \cap N)k = gMg'^{-1}\cap N = k(g'Mg'^{-1} \cap N).$$   
\item [\rm(b)] $|gMg'^{-1}\cap N|=1$ if and only if $gMg^{-1} \cap N = 1 = g'Mg'^{-1} \cap N)$.

\item [\rm(c)] $g \in NM$ if and only if $g'\in NM$ in which case $gMg'^{-1}\cap N = (k(M\cap N)k^{-1})k^*$ where $g=kh,g'=k'h'$ and $k^*=kk'^{-1}$.
\end{enumerate}
\end{lem}
\begin{proof} (a)\qua We have an equality $ghg'^{-1}=k$, where $k$ is our given element of $N$ and $h \in M$. Then $gMg'^{-1} \cap N = kg'h^{-1}Mg'^{-1} \cap N = kg'Mg'^{-1} \cap kN =  k(g'Mg'^{-1} \cap N)$. Similarly  we obtain $(gMg^{-1} \cap N)k = gMg'^{-1}\cap N$.

(b)\qua  This is immediate from (a).

(c)\qua  Let $ g =kh \in NM$, where $k \in N, h \in M$. Then $gMg'^{-1}\cap N = khMg'^{-1}\cap kN = k(Mg'^{-1}\cap N)$.  This means that  $Mg'^{-1}\cap N$ is nonempty and so we have an equality $h'g'^{-1} = k'$ giving $g'=k'h' \in NM$. \end{proof}

From this we can now derive the following corollary to \fullref{thm2}.

\medskip
\textbf{Corollary}\qua {\sl Let $G = \langle X : r=1 \rangle$, where $r$ is cyclically reduced, be a one-relator group and $M = F(S),N = F(T)$ be Magnus subgroups of $G$. For any $g,g' \in G$, one of the following holds:
\begin{enumerate}
\item[\rm(i)] $gMg'^{-1}\cap N$ is empty.

\item[\rm(ii)] $gMg'^{-1}\cap N$ is nonempty, $g, g' \in  NM$ and $gMg'^{-1}\cap N$ is a both a left coset of a conjugate of $M \cap N$ and a right coset of a (different) conjugate of $M \cap N$.

\item[\rm(iii)] $gM{g'}^{-1}\cap N$ is nonempty, $g, g' \notin  NM$ and $gMg'^{-1}\cap N$ is a right coset of the cyclic group $gMg^{-1}\cap N$ and a left coset of the cyclic group $g'Mg'^{-1}\cap N$.
\end{enumerate}  }
\begin{proof} This is immediate from \fullref{lemma6.1} and \fullref{thm2}. \end{proof}

Although the Corollary is formally a slightly more general statement than \fullref{thm2}, the greater generality seems to be of no particular value in making arguments.  One might have hoped that in the analysis of an equality of the form $$gh_0bh_1 \ldots bh_mg^{-1} = k_0bh_1 \ldots bk_m$$ such as that occurring in \fullref{Claim 4.2} -- where $g, h_i, k_i \in G^*$ so that the Normal Form equalities $gh_0z_0^{-1} =k_0, \overleftarrow {z_0}h_1z_1^{-1}=k_1, \ldots, \overleftarrow {z_{n-1}}h_ng^{-1}=k_n$ are all of the form described in the Corollary relative to the Magnus subgroups $M=F(A^*,B^*)$ and $N=F(B^*,C^*)$ of $G^*$ -- would permit a direct inductive argument taking the statement of the Corollary as the inductive hypothesis.  However, this does not seem to be possible, probably because the Corollary is obtained so easily and so the level of additional generality is thus very slight. 

\bibliographystyle{gtart}
\bibliography{link}

\end{document}